\documentclass[smallextended,envcountsect]{svjour3}
\smartqed
\usepackage{verbatim}
\usepackage[T1]{fontenc}
\usepackage{mathptmx,latexsym,amssymb,amsmath,url,xspace,lineno,color,enumerate}
\usepackage[width={6.5in}]{geometry}

\newcommand{\RR}{\mathbb R}
\newcommand{\NN}{\mathbb N}

\renewcommand{\leq}{\leqslant}

\renewcommand{\geq}{\geqslant}

\journalname{}

\begin{document}

\title{Nonhomogeneous Hemivariational Inequalities with Indefinite Potential and Robin Boundary Condition}
\subtitle{}
\titlerunning{Nonhomogeneous Hemivariational Inequalities with Indefinite Potential}
\author{Nikolaos S. Papageorgiou,  Vicen\c tiu D. R\u{a}dulescu,  Du\v san~D.~Repov\v s}
\institute{N.S. Papageorgiou \at Department of Mathematics, National Technical University,
				Zografou Campus, Athens 15780, Greece\\
              \email{npapg@math.ntua.gr}      
 \and
           V.D. R\u{a}dulescu,   \at
              Department of Mathematics, Faculty of Sciences, King Abdulaziz University, P.O. Box 80203, Jeddah 21589, Saudi
              Arabia\\
              Department of Mathematics, University of Craiova, Street A.I. Cuza No 13, 200585 Craiova, Romania\\
              \email{vicentiu.radulescu@imar.ro; vicentiu.radulescu@math.cnrs.fr}
 \and D.D. Repov\v s \at Faculty of Education and Faculty of Mathematics and Physics,
University of Ljubljana, SI-1000 Ljubljana, Slovenia \\
              \email{dusan.repovs@guest.arnes.si}
}
\date{} 

\maketitle

\begin{abstract}
We consider a nonlinear, nonhomogeneous Robin problem with an indefinite potential and a nonsmooth primitive in the reaction term. In fact, the right-hand side of the problem (reaction term) is the Clarke subdifferential of a locally Lipschitz integrand. We assume that asymptotically this term is resonant with respect the principal eigenvalue (from the left). We prove the existence of three nontrivial smooth solutions, two of constant sign and the third nodal. We also show the existence of extremal constant sign solutions. The tools come from nonsmooth critical point theory and from global optimization (direct method).
\end{abstract}

\section{Introduction}
In this paper, we study a class of nonlinear elliptic partial differential inclusions with Robin boundary condition and involving a nonhomogeneous differential operator. The resulting inclusion is known in the literature as  {a} ``hemivariational inequality". Hemivariational inequalities were introduced as an extension of the classical variational inequalities, in order to deal with problems of mechanics and engineering in which the relevant energy functionals are neither convex nor smooth (the so-called super-potentials). Many such applications can be found in the book  {by}  Panagiotopoulos [1].

Our aim is to prove a ``three solutions theorem",  {which provides} regularity and sign information for all of them.
Such multiplicity results were proved by Liu [2], Liu and Liu [3], Papageorgiou and Papageorgiou [4] for Dirichlet problems driven by the $p$-Laplacian with zero potential and with a smooth primitive. However, none of the aforementioned works allows for resonance to occur and they do not produce nodal (that is, sign changing) solutions. Multiple nontrivial smooth solutions for Neumann $p$-Laplacian hemivariational inequalities were obtained by Aizicovici, Papageorgiou and Staicu [5,6]. In [5] the potential function $\xi\equiv 0$ and the multivalued nonlinearity  is crossing, but the situation is complementary to the one studied here. The authors, using degree theory techniques,
 produce two nontrivial smooth solutions with no sign information. In [6], the potential $\xi\equiv\xi_0\in(0,+\infty)$ and the multivalued reaction is $(p-1)$-superlinear. Finally, we  {also mention recent work} on nonlinear Neumann and Robin problems with a smooth primitive, by Papageorgiou and R\u adulescu [7, 8],
   {Marano and Papageorgiou [9, 10], Papageorgiou and R\u adulescu [11], and Papageorgiou and Winkert [12]}.

\section{Statement of the Problem}
Let $\Omega\subseteq \RR^N$ be a bounded domain with a $C^2$-boundary $\partial\Omega$. In this paper we study the following nonlinear elliptic partial differential inclusion:
\begin{equation}\label{eq1}
	\left\{\begin{array}{ll}
		-{\rm div}\,a(Du(z))+\xi(z)|u(z)|^{p-2}u(z)\in\partial F(z,u(z))\ \mbox{in}\ \Omega, \\
	\displaystyle	\frac{\partial u}{\partial n_{a}}+\beta(z)|u(z)|^{p-2}u(z)=0\ \mbox{on}\ \partial\Omega.
	\end{array}\right\}
\end{equation}

In this inclusion, the map $a:\RR^N\rightarrow \RR^N$ involved in the definition of the differential operator, is a strictly monotone, continuous map which satisfies certain other regularity and growth conditions listed in hypotheses $H(a)$ (see Section \ref{sec2}). These hypotheses are general enough to incorporate in our framework many differential operators of interest, such as the $p$-Laplacian. We stress that in our case the differential operators need not be homogeneous. The potential function $\xi\in L^{\infty}(\Omega)$ is in general, sign changing (indefinite potential). In the reaction term (right-hand side of (\ref{eq1})), $F(z,x)$ is a real valued function on $\Omega\times\RR$ which is measurable in $z\in\Omega$ for every $x\in\RR$ and locally Lipschitz in $x\in\RR$ for $\mu$-a.a. $z\in\Omega$. By $\partial F(z,x)$ we denote the generalized subdifferential in the sense of Clarke (see Section \ref{sec2}).

 In the boundary condition, $\frac{\partial u}{\partial n_a}$ denotes the generalized normal derivative defined by extension of the map
$$C^1(\overline{\Omega})\ni u\mapsto\frac{\partial u}{\partial_{n_a}}=(a(Du),n)_{\RR^N}$$
with $n(\cdot)$ being the outward unit normal on $\partial\Omega$. This kind of normal derivative is dictated by the nonlinear Green's identity (see, for example, Gasinski and Papageorgiou [13], p. 210).

In this work, we assume that $\partial F(z,\cdot)$ exhibits sublinear growth as $x\rightarrow\pm\infty$ and in the special case of the $p$-Laplace differential operator, the multivalued quotient $\frac{\partial F(z,x)}{|x|^{p-2}x}$ asymptotically as $x\rightarrow\pm\infty$ stays below the principal eigenvalue $\hat{\lambda}_1$ of the differential operator $u\mapsto -\Delta_p u+\xi(z)|u|^{p-2}u,\ u\in W^{1,p}(\Omega)$ with  {the} Robin boundary condition. Here, $\Delta_p$ denotes the $p$-Laplace differential operator defined by
$\Delta_p u={\rm div}\,(|Du|^{p-2}Du)$ for all $u\in W^{1,p}(\Omega)$, $1<p<\infty$.

In fact, we allow for full interaction (resonance) with $\hat{\lambda}_1$. So, the problem is resonant. The resonance occurs from the left of $\hat{\lambda}_1$ and this makes the energy (Euler) functional of the problem coercive. This  {allows for} the use of global optimization techniques (direct method of the calculus of variations) in order to obtain solutions of constant sign. Near the origin, our conditions on $\partial F(z,\cdot)$ are such that, again in the special case of the $p$-Laplace differential operator, they imply that the quotient $\frac{\partial F(z,x)}{|x|^{p-2}x}$ stays above $\hat{\lambda}_1$. So, we can say that we have a ``crossing" multivalued reaction term, since the quotient $\frac{\partial F(z,x)}{|x|^{p-2}x}$ crosses at least $\hat{\lambda}_1$ as we move from $x=0$ to $x=\pm\infty$.

\section{Mathematical Background}\label{sec2}

Let $X$ be a Banach space and $X^*$ its topological dual. By $\langle\cdot,\cdot\rangle$ we denote the duality brackets for the pair $(X^*,X)$. We say that $\varphi:X\rightarrow\RR$ is locally Lipschitz, if for every $x\in X$ we can find a neighbourhood $U(x)$ of $x$ and a constant $k(x)>0$ such that
\begin{equation}\label{eq2}
	|\varphi(u)-\varphi(v)|\leq k(x)||u-v||\ \mbox{for all}\ u,v\in U(x).
\end{equation}

If (\ref{eq2}) is satisfied for all $u,v\in X$ and with $k(x)=k>0$ independent of $x\in X$, then we have the usual Lipschitz continuous function. Note that, if $\varphi:X\rightarrow\RR$ is Lipschitz continuous on every bounded set in $X$, then $\varphi$ is locally Lipschitz. Moreover, if $X$ is finite dimensional, then the converse is also true. Finally, if $\varphi:X\rightarrow\RR$ is continuous, convex or if $\varphi\in C^1(X,\RR)$, then $\varphi$ is locally Lipschitz.

Given a locally Lipschitz function, the ``generalized directional derivative" of $\varphi$ at $x\in X$ in the direction $h\in X$, denoted by $\varphi^0(x;h)$, is defined by
$$\varphi^0(x;h)=\limsup\limits_{\substack{x'\rightarrow x \\\lambda\rightarrow 0^+}}\frac{\varphi(x'+\lambda h)-\varphi(x')}{\lambda}\,.$$

It is easy to see that
\begin{itemize}
	\item [(i)] $h\mapsto\varphi^0(x;h)$ is sublinear and Lipschitz continuous;
	\item [(ii)] $(x,h)\mapsto\varphi^0(x;h)$ is upper semicontinuous on $X\times X$;
	\item [(iii)] $\varphi^0(x;-h)=(-\varphi)^0(x;h)$ for all $x,h\in X$.
\end{itemize}

So, $\varphi^0(x;\cdot)$ is the support function of a nonempty, convex and $w^*$-compact set defined by
$$\partial\varphi(x)=\{ x^*\in X^*:\langle x^*,h \rangle \leq \varphi^0(x;h)\ \mbox{for all}\ h\in X \}.$$

The multifunction $x\mapsto\partial\varphi(x)$ is known as the generalized or Clarke subdifferential of $\varphi$. If $\varphi\in C^1(X,\RR)$, then as we  {have} already mentioned, $\varphi$ is locally Lipschitz and
$\partial\varphi(x)=\{ \varphi'(x) \}.$
Similarly, if $\varphi:X\rightarrow\RR$ is continuous convex, then $\varphi$ is locally Lipschitz and the generalized subdifferential coincides with the subdifferential in the sense of convex analysis defined by
$\partial_{c}\varphi(x)=\{ x^*\in X^*:\langle x^*,h\rangle \leq \varphi(x+h)-\varphi(x)\ \mbox{for all}\ h\in X \}.$

For locally Lipschitz functions $\varphi,\psi:X\rightarrow\RR$ and $\lambda\in\RR$ we have
\begin{itemize}
	\item [(i)] $\partial(\varphi+\psi)(x)\subseteq\partial\varphi(x)+\partial\psi(x)$ for all $x\in X$ (equality holds if one of them is a singleton);
	\item [(ii)] $\partial(\lambda\varphi)(x)=\lambda\partial\varphi(x)$ for all $x\in X$.
\end{itemize}

The multifunction $\partial\varphi:X\rightarrow 2^{X^*}\setminus\emptyset$ is upper semicontinuous (usc for short) from $X$,  {equipped} with the norm topology into $X^*$, furnished with the $w^*$-topology. This implies that ${\rm Gr}\,\partial\varphi=\{ (x,x^*)\in X\times X^*: x^*\in \partial\varphi(x)\}$ is closed in $X\times X^*_{w^*}$. For a complete presentation of the subdifferential theory for locally Lipschitz functions, we refer to Clarke [14].

Using the Clarke subdifferential theory, we can have a nonsmooth critical point theory extending the classical theory for $C^1$-functions. So, let $\varphi:X\rightarrow\RR$ be a locally Lipschitz function. We say that $x\in X$ is a ``critical point" of $\varphi$, if $0\in \partial\varphi(x)$. By $K_\varphi$ we denote the set of critical points of $\varphi$. Let
\begin{equation}\label{eq3}
	m_\varphi(x)=\inf\left[||x^*||_*:x^*\in \partial\varphi(x)\right].
\end{equation}

We say that $\varphi$ satisfies the ``nonsmooth  {PS}-condition", if the following property holds.
\begin{center}
``Every sequence $\{ x_n \}_{n\geq 1}\subseteq X$ such that\\
$\{\varphi(x_n) \}_{n\geq 1}\subseteq\RR$\ \mbox{is bounded}\
 \mbox{and} $m_\varphi(x_n)\rightarrow 0$\ \mbox{as}\ $n\rightarrow\infty$,\\
admits a strongly convergent subsequence".
\end{center}

Evidently, this notion extends the classical $C$-condition for $C^1$-functions.

Using this compactness-type condition on the functional $\varphi$, one can prove a deformation theorem from which follows the nonsmooth minimax theory of the critical points of $\varphi$. A basic result in that theory is the so-called ``mountain pass theorem" (see Chang [15], Gasinski and Papageorgiou [16], R\u adulescu [17]).

\begin{theorem}\label{th1}
	Let $\varphi:X\rightarrow\RR$ be a locally Lipschitz function which satisfies the nonsmooth C-condition. Assume that there exist $u_0,u_1\in X$ and $r>0$ with $ ||u_1-u_0||>r$,
	$\max\{\varphi(u_0),\varphi(u_1)\}<\inf[\varphi(u):||u-u_0||=r]=m_r$,
and $c=\inf\limits_{\gamma\in\Gamma}\max\limits_{0\leq t\leq1}\varphi(\gamma(t))$ with $\Gamma=\{\gamma\in C([0,1],X):\gamma(0)=u_0,\gamma(1)=u_1\}$. Then $c\geq m_r$ and $c$ is a critical value of $\varphi$ (that is, we can find $u\in K_\varphi$ such that $\varphi(u)=c$).
\end{theorem}

In the analysis of problem (\ref{eq1}) we will use the following spaces:
\begin{itemize}
	\item [-] the Sobolev space $W^{1,p}(\Omega)$;
	\item [-] the ordered Banach space $C^1(\overline{\Omega})$;
	\item [-] the boundary Lebesgue space $L^p(\partial\Omega)$.
\end{itemize}

In what follows, we denote by $||\cdot||$  the norm of the Sobolev space $W^{1,p}(\Omega)$, defined by
$||u||=\left[ ||u||^p_p+||Du||^p_p \right]^{^1/p}$ for all $u\in W^{1,p}(\Omega).$
The positive (order) cone for the ordered Banach space $C^1(\overline{\Omega})$ is given by
$C_+=\{ u\in C^1(\overline{\Omega}):u(z)\geq0\ \mbox{for all}\ z\in\overline{\Omega} \}.$
This cone has a nonempty interior, given by
$D_+=\{ u\in C_+:u(z)>0\ \mbox{for all}\ z\in\overline{\Omega} \}.$

On $\partial\Omega$ we consider the $(N-1)$-dimensional Hausdorff (surface) measure $\sigma(\cdot)$. Using this measure, we can define in the usual way the Lebesgue spaces $L^q(\partial\Omega)$, $1\leq q\leq\infty$. The theory of Sobolev spaces says that there exists a unique continuous linear map $\gamma_0:W^{1,p}(\Omega)\rightarrow L^p(\partial\Omega)$, known as the ``trace map", such that
$\gamma_0(u)=u|_{\partial\Omega}$ for all $u\in W^{1,p}(\Omega)\cap C(\overline{\Omega}).$
In this way we extend the notion of boundary values to all Sobolev functions. The trace map is compact into $L^q(\partial\Omega)$ for all $q\in\left[1,\frac{Np-p}{N-p}\right)$ if $N\geq3$, and into $L^q(\partial\Omega)$ for all $q\in[1,\infty)$ if $N=1,2$. In the sequel, for the sake of notational simplicity, we drop the use of the map $\gamma_0$. All restrictions of Sobolev functions on $\partial\Omega$ are understood in the sense of traces. Finally, we mention that the trace map is not surjective on $L^p(\partial\Omega)$. We have
$${\rm im}\, \gamma_0=W^{\frac{1}{p'},p}(\partial\Omega)\ \mbox{with}\ \frac{1}{p}+\frac{1}{p'}=1.$$
Moreover, we have
${\rm ker}\,\gamma_0=W^{1,p}_0(\Omega).$

Let $d\in C^1(0,\infty)$ with $d(t)>0$ for all $t>0$ and assume that there exist $c_1,\ c_2>0$ such that
\begin{equation}\label{eq4}
	0<\hat{c}\leq\frac{td'(t)}{d(t)}\leq c_0,\ c_1t^{p-1}\leq d(t)\leq c_2(1+t^{p-1})\quad\mbox{for all $t>0$}.
\end{equation}

Our hypotheses on the map $a(\cdot)$, involved in the definition of the differential operator in problem (\ref{eq1}), are the following:

\smallskip
$H(a):$ $a(y)=a_0(|y|)y$ for all $y\in \RR^N$ with $a_0(t)>0$ for all $t>0$ and
\begin{itemize}
	\item [(i)] $a_0\in C^1(0,\infty),t\mapsto t a_0(t)$ is strictly increasing on $(0,+\infty),ta_0(t)\rightarrow 0^+$ as $t\rightarrow 0^+$ and
	$$\lim\limits_{t\rightarrow 0^+}\frac{ta'_0(t)}{a_0(t)}>-1;$$
	\item [(ii)] $|\nabla a(y)|\leq c_3\frac{d(y)}{|y|}$ for all $y\in \RR^N\backslash\{0\}$ and some $c_3>0$;
	\item [(iii)] $(\nabla a(y)h,h)_{\RR^N}\geq \frac{d(|y|)}{|y|}|h|^2$ for all $y\in \RR^N\backslash \{0\}$, $h\in \RR^N$;
	\item [(iv)] if $G_0(t)=\int^t_0 sa_0(s)ds$ (for $t>0$), then there exists $q\in(1,p]$ such that
	\begin{eqnarray*}
		&& t\mapsto G_0(t^{1/q})\ \mbox{is convex on}\ (0,\infty);\\
		&& \lim\limits_{t\rightarrow 0^+}\frac{qG_0(t)}{t^q}=\tilde{c}>0.
	\end{eqnarray*}
\end{itemize}
\begin{remark}
	Hypotheses $H(a)(i),(ii),(iii)$ are designed so that we can use the nonlinear regularity theory of Lieberman [18] and the nonlinear maximum principle of Pucci and Serrin [19], pp. 111, 120. Hypothesis $H(a)(iv)$ serves the particular needs of our problem but it is very mild and it is satisfied in most cases of interest, as the examples which follow illustrate.
\end{remark}

From hypotheses $H(a)$ it follows that $t\mapsto G_0(t)$ is strictly convex and strictly increasing. We set $G(y)=G_0(|y|)$ for all $y\in \RR^N$. Evidently, $G(\cdot)$ is convex and $G(0)=0$. We have
$$\nabla G(y)=G'_0(|y|)\frac{y}{|y|}=a_0(|y|)y=a(y)\ \mbox{for all}\ y\in \RR^N\backslash\{0\},\ \nabla G(0)=0.$$

Therefore $G(\cdot)$ is the primitive of $a(\cdot)$. The convexity of $G(\cdot)$ and since $G(0)=0$, imply that
\begin{equation}\label{eq5}
	G(y)\leq(a(y),y)_{\RR^N}\ \mbox{for all}\ y\in\RR^N.
\end{equation}

Hypotheses $H(a)(i),(ii),(iii)$ and (\ref{eq4}), (\ref{eq5}) lead to the following lemma which summarizes the main properties of the map $a(\cdot)$.
\begin{lemma}\label{lem2}
	If hypotheses $H(a)(i),(ii),(iii)$ hold, then
	\begin{itemize}
		\item [$(a)$] the map $y\mapsto a(y)$ is continuous and strictly monotone (hence maximal monotone, too);
		\item [$(b)$] $|a(y)|\leq c_4(1+|y|^{p-1})$ for all $y\in\RR^N$ and some $c_4>0$;
		\item [$(c)$] $(a(y),y)_{\RR^N}\geq \frac{c_1}{p-1}|y|^p$ for all $y\in \RR^N$.
	\end{itemize}
\end{lemma}

This lemma and (\ref{eq5}), give the following growth estimates for the primitive $G(\cdot)$.

\begin{corollary}\label{cor3}
	If hypotheses $H(a)(i),(ii),(iii)$ hold, then $\frac{c_1}{p(p-1)}|y|^p\leq G(y)\leq c_5(1+|y|^p)\ \mbox{for all}\ y\in \RR^N\ \mbox{and some}\ c_5>0.$
\end{corollary}

{\it Examples}. The following maps satisfy hypotheses $H(a)$ above (see [8]).
\begin{itemize}
	\item [(a)] $a(y)=|y|^{p-2}$ with $1<p<\infty$.
		Then the differential operator is the $p$-Laplacian
		$\Delta_p u = {\rm div}\,(|Du|^{p-2}Du)$ for all $u\in W^{1,p}(\Omega).$
	\item [(b)] $a(y)=|y|^{p-2}y+|y|^{q-2}y$ with $1<q<p<\infty$.
		Then the differential operator is the $(p,q)$-Laplacian
		$\Delta_p u+\Delta_q u$ for all $u\in W^{1,p}(\Omega).$

	Such operators arise in problems of mathematical physics, see Cherfils and Ilyasov [20]. Recently, some existence and multiplicity results for such equations have been established. All these problems are with no potential (that is, $\xi\equiv0$) and with  {a} smooth primitive $F(z,\cdot)$. We mention the works of Aizicovici, Papageorgiou and Staicu [21], Cingolani and Degiovanni [22], Papageorgiou and R\u adulescu [23],  Papageorgiou, R\u adulescu and Repov\v s [24], Sun [25], Sun, Zhang and Su [26].
	\item [(c)] $a(y)=(1+|y|^2)^{\frac{p-2}{2}}y$ with $1<p<\infty$.
		Then the differential operator is the generalized $p$-mean curvature operator
		${\rm div}\,( \left(1+|Du|^2\right)^{\frac{p-2}{2}} Du )$ for all $u\in W^{1,p}(\Omega).$
	\item [(d)] $a(y)=|y|^{p-2}y+\frac{|y|^{p-2}y}{1+|y|^p}$ with $1<p<\infty$.
\end{itemize}

The hypotheses on the potential function $\xi(\cdot)$ and the boundary coefficient are the following:
\begin{itemize}
	\item [$H(\xi):$] $\xi\in L^\infty(\Omega).$
	\item [$H(\beta):$] $\beta\in C^{0,\alpha}(\partial\Omega)$ with $\alpha\in(0,1)$ and $\beta(z)\geq0$ for all $z\in\partial\Omega.$
\end{itemize}
\begin{remark}
	When $\beta\equiv0$, we have the usual Neumann problem.
\end{remark}
	
	Suppose that $F_0:\Omega\times\RR\rightarrow\RR$ is a locally Lipschitz integrand satisfying
	$$|F_0(z,x)|\leq a_0(z)(1+|x|^p)\ \mbox{for almost all}\ z\in\Omega,\ \mbox{for all}\ x\in\RR,\ \mbox{with}\ a_0\in L^\infty(\Omega).$$
	Consider the functional $\varphi_0:W^{1,p}(\Omega)\rightarrow\RR$ defined by
	$$\varphi_0(u)=\int_\Omega G(Du)dz+\frac{1}{p}\int_\Omega\xi(z)|u|^pdz+\frac{1}{p}\int_{\partial\Omega}\beta(z)|u|^pd\sigma-\int_\Omega F_0(z,u)dz\quad\mbox{for all $u\in W^{1,p}(\Omega)$}.$$

Then $\varphi_0$ is locally Lipschitz (see Clarke [14]). From Gasinski and Papageorgiou [27] (see also Papageorgiou and R\u adulescu [28] for the critical case), we have the following property.

\begin{proposition}\label{prop4}
	Let $u_0\in W^{1,p}(\Omega)$ be a local $C^1(\overline{\Omega})$-minimizer of $\varphi_0$, that is, there exists $\rho_0>0$ such that
	$$\varphi_0(u_0)\leq\varphi_0(u_0+h)\ \mbox{for all}\ h\in C^1(\overline{\Omega})\ \mbox{with}\ ||h||_{C^1(\overline{\Omega})}\leq\rho_0.$$
	Then $u_0\in C^{1,\alpha}(\overline{\Omega})$ with $0<\alpha<1$ and $u_0$ is also a local $W^{1,p}(\Omega)$-minimizer of $\varphi_0$, that is, there exists $\rho_1>0$ such that
	$$\varphi_0(u_0)\leq\varphi_0(u_0+h)\ \mbox{for all}\ h\in W^{1,p}(\Omega)\ \mbox{with}\ ||h||\leq\rho_1.$$
\end{proposition}

	Let $A:W^{1,p}(\Omega)\rightarrow W^{1,p}(\Omega)^*$ be the nonlinear map defined by
	$$\langle A(u),h \rangle=\int_\Omega(a(Du),Dh)_{\RR^N}dz\ \mbox{for all}\ u,h\in W^{1,p}(\Omega).$$

From Gasinski and Papageorgiou [29] (Problem 2.192), we have:
\begin{proposition}\label{prop5}
	If hypotheses $H(a)$ hold, then $A:W^{1,p}(\Omega)\rightarrow W^{1,p}(\Omega)^*$ is continuous, monotone (hence maximal monotone, too) and of type $(S)_+$, that is,
	$$``u_n\xrightarrow{w}u\ \mbox{in}\ W^{1,p}(\Omega),\limsup\limits_{n\rightarrow\infty} \langle A(u_n),u_n-u\rangle \leq0\Rightarrow u_n\rightarrow u\ \mbox{in}\ W^{1,p}(\Omega)."$$
\end{proposition}

We will also need some basic facts about the spectrum of the nonlinear operator $u\mapsto -\Delta_p u+\xi(z)|u|^{p-2}u$ with  {the} Robin boundary condition. So, we consider the following nonlinear eigenvalue problem:
\begin{equation}\label{eq6}
	\left\{ \begin{array}{ll}
		-\Delta_p u(z)+\xi(z)|u(z)|^{p-2}u(z)=\hat{\lambda}|u(z)|^{p-2}u(z)\ \mbox{in}\ \Omega, \\
	\displaystyle	\frac{\partial u}{\partial n_p}+\beta(z)|u|^{p-2}u=0\ \mbox{on}\ \partial\Omega,
	\end{array} \right\}
\end{equation}
where
 $\frac{\partial u}{\partial n_p}=|Du|^{p-2}(Du,n)_{\RR^N}$. We say that $\hat{\lambda}\in\RR$ is an eigenvalue if problem (\ref{eq6}) admits a nontrivial solution $\hat{u}\in W^{1,p}(\Omega)$, known as an eigenfunction corresponding to $\hat{\lambda}$.

From Mugnai and Papageorgiou [30] (Neumann problem) and Papageorgiou and R\u a\-du\-les\-cu [31] (Robin problem), we know that  {problem (\ref{eq6})} has a smallest eigenvalue $\hat{\lambda}_1(p,\xi,\beta)\in\RR$ which has the following properties:
\begin{itemize}
	\item [] $\hat{\lambda}_1(p,\xi,\beta)$ is isolated in the spectrum $\hat{\sigma}(p,\xi,\beta)$ of \eqref{eq6} (that is, we can find $\varepsilon>0$ such that $(\hat{\lambda}_1(p,\xi,\beta),\hat{\lambda}_1(p,\xi,\beta)+\varepsilon)\cap\hat{\sigma}(p,\xi,\beta)=\emptyset$);
	\item [] $\hat{\lambda}_1(p,\xi,\beta)$ is simple (that is, if $\hat{u},\hat{v}$ are two eigenfunctions corresponding to this eigenvalue, then $\hat{u}=\eta\hat{v}$ with $\eta\in\RR\backslash\{0\}$);
	\begin{equation}\label{eq7}
		 \hat{\lambda}_1(p,\xi,\beta)=\inf\left[ \frac{k_p(u)}{||u||^p_p}:u\in W^{1,p}(\Omega),u\neq0 \right],\hspace{5.1cm}
	\end{equation}
	with $k_p:W^{1,p}(\Omega)\rightarrow\RR$ being the $C^1$-functional defined by \\ $k_p(u)=||Du||^p_p+\int_\Omega\xi(z)|u|^pdz+\int_{\partial\Omega}\beta(z)|u|^pd\sigma .$
\end{itemize}

In (\ref{eq7}) the infimum is realized on the one-dimensional eigenspace, corresponding to the eigenvalue $\hat{\lambda}_1(p,\xi,\beta)$.  It follows from the above properties that the elements of this eigenspace do not change  {the} sign. By $\hat{u}_1(p,\xi,\beta)$ we denote the $L^p$-normalized (that is, $||\hat{u}_1(p,\xi,\beta)||_p=1$) positive eigenfunction. The nonlinear regularity theory of Lieberman [18] implies that $\hat{u}_1(p,\xi,\beta)\in C_+$. In fact, the nonlinear maximum principle (see, for example, Gasinski and Papageorgiou [13], p. 738) implies that $\hat{u}_1(p,\xi,\beta) \in D_+$. The eigenfunctions corresponding to eigenvalues $\hat{\lambda}\neq\hat{\lambda}_1(p,\xi,\beta)$ are nodal (that is, sign changing). The spectrum $\hat{\sigma}(p,\xi,\beta)$ is closed and $\hat{\lambda}_1(p,\xi,\beta)$ is isolated. So, the second eigenvalue $\hat{\lambda}_2(p,\xi,\beta)$ is well-defined by
$	\hat{\lambda}_2(p,\xi,\beta) = \min\{\hat{\lambda}\in\hat{\sigma}(p,\xi,\beta) : \hat{\lambda}>\hat{\lambda}_1(p,\xi,\beta)\}$.

Additional eigenvalues can be generated  {by} using the Ljusternik-Schnirelmann minimax scheme. In this way, we produce a strictly increasing sequence $\{\hat{\lambda}_k\}_{k\in\NN}$ of eigenvalues such that $\hat{\lambda}_k\rightarrow+\infty$. These are known as the ``LS-eigenvalues" of (\ref{eq6}) and we do not know if they exhaust $\hat{\sigma}(p,\xi,\beta)$. We know that they exhaust  {$\hat{\sigma}(p,\xi,\beta)$} if $p=2$ (linear eigenvalue problem) or if $N=1$ (ordinary differential equation). A variational characterisation of $\hat{\lambda}_2(p,\xi,\beta)$ can be obtained using the minimax expression of the Ljusternik-Schnirelmann minimax scheme. There is an alternative minimax characterization which  is more convenient for our purpose. So, we define
\begin{align*}
	& \partial B_1^{L^p} = \{u\in L^p(\Omega) : \|u\|_p=1\},\\
	& M = W^{1,p}(\Omega)\cap\partial B_1^{L^p},\\
	& \hat{\Gamma} = \{\hat{\gamma}\in C([-1,1],M) : \hat{\gamma}(-1) = -\hat{u}_1(p,\xi,\beta),\ \hat{\gamma}(1) = \hat{u}_1(p,\xi,\beta)\}.
\end{align*}

From [30] and [31],  we have the following alternative minimax characterization of $\hat{\lambda}_2(p,\xi,\beta)$.
\begin{proposition}\label{prop6}
	$\hat{\lambda}_2(p,\xi,\beta) = \inf\limits_{\hat{\gamma}\in\hat{\Gamma}} \max\limits_{-1\leqslant t\leqslant 1} k_p(\hat{\gamma}(t))$.
\end{proposition}

Finally, let us fix our notation. For $x\in\RR$ we set $x^\pm=\max\{\pm x,0\}$ and for $u\in W^{1,p}(\Omega)$ we define $u^\pm(\cdot) = u(\cdot)^\pm$. We have
$u^\pm \in W^{1,p}(\Omega),\ u = u^+ - u^-, \ |u| = u^+ + u^-.$
For $u, v \in W^{1,p}(\Omega),\ v\leqslant u$, we introduce the order interval
$[v,u] = \{y\in W^{1,p}(\Omega) : v(z)\leqslant y(z)\leqslant u(z)\ \mbox{for almost all}\ z\in\Omega\}.$

\section{Solutions of Constant Sign}

In this section we produce two nontrivial, constant sign smooth solutions. These solutions are obtained by global optimization of suitable truncations and perturbations of the energy functional. In addition, we establish the existence of extremal constant sign solutions of (\ref{eq1}), that is, we show that problem (\ref{eq1}) has a smallest positive solution and a  {biggest} negative solution. These extremal constant sign solutions are crucial in obtaining a nodal (sign changing) solution in Section 4.

The hypotheses on the nonsmooth primitive $F(z,x)$ are the following:

\smallskip
$H_1:$  $F:\Omega\times\RR\rightarrow\RR$ is a locally Lipschitz integrand (that is, $z\mapsto F(z,x)$ is measurable and for almost all $z\in\Omega,\ x\mapsto F(z,x)$ is  locally Lipschitz) such that $F(z,0) = 0$ for almost all $z\in\Omega$ and
\begin{itemize}
	\item [(i)] $|v|\leqslant a(z)(1+|x|^{p-1})$ for almost all $z\in\Omega$  {and for all} $x\in\RR$, $v\in\partial F(z,x)$, with $a\in L^\infty(\Omega)_+$;
	\item [(ii)] $\limsup_{x\rightarrow\pm\infty} \frac{v}{|x|^{p-2}x} \leqslant \frac{c_1}{p-1} \hat{\lambda}_1(p,\hat{\xi},\hat{\beta})$ uniformly for almost all $z\in\Omega$  {and for all} $v\in\partial F(z,x)$, with $c_1>0$ as in (\ref{eq4}) and with $\hat{\xi}=\frac{p-1}{c_1}\xi$, $\hat{\beta}=\frac{p-1}{c_1}\beta$;
	\item [(iii)] $\lim_{x\rightarrow\pm\infty}[vx-pF(z,x)] = +\infty$ uniformly for almost all $z\in\Omega$  {and for all} $v\in\partial F(z,x)$;
	\item [(iv)] there exists a function $\vartheta_0\in L^\infty(\Omega)$ such that
		$$\tilde{c}\hat{\lambda}_1(q, \tilde{\xi}^+,\tilde{\beta}) \leqslant \vartheta_0(z)\ \mbox{for almost all}\ z\in\Omega,\ \vartheta_0\not\equiv\tilde{c}\hat{\lambda}_1(q, \tilde{\xi}^+, \tilde{\beta})$$
		$$\vartheta_0(z) \leqslant \liminf_{x\rightarrow 0}\frac{v}{|x|^{q-2}x}\ \mbox{uniformly for almost all}\ z\in\Omega\ \mbox{, all}\ v\in\partial F(z,x)$$
		with $\tilde{c}>0$ and $q\in(1,p]$ as in hypothesis $H(a)(iv)$ and $\tilde{\xi}=\frac{1}{\tilde{c}}\xi$, $\tilde{\beta}=\frac{1}{\tilde{c}}\beta$.
\end{itemize}
\begin{remark}
	If $a(y)=|y|^{p-2}y$ for all $y\in\RR^N$ (the case of the $p$-Laplace differential operator), then $c_1 = p-1$ (see (\ref{eq4})). So, $\hat{\xi} = \xi$, $\hat{\beta}=\beta$ (see hypothesis $H_1(ii)$). Hence in this special case, we see that hypothesis $H_1(ii)$ incorporates in our framework problems which  are resonant at $\pm\infty$ with respect to the principal eigenvalue of the differential operator $u\mapsto -\Delta_pu+\xi(z)|u|^{p-2}u$ with Robin boundary condition. Hypothesis $H_1(iii)$ says that this resonance occurs from the left of $\hat{\lambda}_1(p,\hat{\xi},\hat{\beta})$ in the sense that
$$\frac{c_1}{p-1}\hat{\lambda}_1(p,\hat{\xi},\hat{\beta})|x|^p - pF(z,x)\rightarrow+\infty\ as\ x\rightarrow\pm\infty,\ \mbox{uniformly for almost all}\ z\in\Omega\,.$$
\end{remark}

This fact makes the energy (Euler) functional of the problem coercive and so techniques of global optimization can be used. To  {better understand} hypothesis $H_1(iv)$ it is again helpful to see the situation in the special case of the $p$-Laplacian (that is, $a(y)=|y|^{p-2}y$). Then we have $q=p$ and $\tilde{c}=1$ (see hypothesis $H(a)(iv)$). Hence we see that hypothesis $H_1(iv)$ implies that as $x\rightarrow0$, the quotient $\frac{v}{|x|^{p-2}x}$ stays above $\hat{\lambda}_1(p,\xi,\beta)$ and so we have a crossing reaction term. In fact, hypothesis $H_1(iv)$ permits also the presence of a concave (that is, of a $(p-1)$-superlinear) term near zero.

Let $\mu>\|\xi\|_\infty$ (see hypothesis $H(\xi)$) and consider the following truncations-perturbations of the primitive $F(z,x)$:
\begin{eqnarray}\label{eq8}
	&&\hat{F}_+(z,x) = \left\{\begin{array}{ll}
		\mbox{0} &\mbox{if}\ x\leqslant0\\
		F(z,x) + \frac{\mu}{p}|x|^p\ &\mbox{if}\ 0<x
	\end{array}\right.
	\mbox{and}\\
	&&\hat{F}_-(z,x) = \left\{\begin{array}{ll}
		F(z,x) + \frac{\mu}{p}|x|^p\ &\mbox{if}\ x<0\\
		\mbox{0} &\mbox{if}\ 0\leqslant x.
	\end{array}\right. \nonumber
\end{eqnarray}

Both $\hat{F}_\pm(z,x)$ are locally Lipschitz integrands and we have
\begin{eqnarray}\label{eq9}
	&&\partial\hat{F}_+(z,x)\subseteq \left\{\begin{array}{ll}
		\mbox{0} &\mbox{if}\ x<0\\
		\{r\partial F(z,0) : 0\leqslant r\leqslant1\} &\mbox{if}\ x=0\\
		\partial F(z,x) + \mu x^{p-1} &\mbox{if}\ 0<x
	\end{array}\right.
	\mbox{and} \\
	&&\partial \hat{F}_-(z,x)\subseteq \left\{\begin{array}{ll}
		\partial F(z,x) + \mu|x|^{p-2}x &\mbox{if}\ x<0\\
		\{r\partial F(z,0) : 0\leqslant r\leqslant1\} &\mbox{if}\ x=0\\
		\mbox{0} &\mbox{if}\ 0<x
	\end{array}\right.\nonumber
\end{eqnarray}
(see Clarke [14], p. 42). Then we introduce the locally Lipschitz functionals $\hat{\varphi}_\pm:W^{1,p}(\Omega)\rightarrow\RR$ defined by
$$\hat{\varphi}_\pm(u) = \int_\Omega G(Du)dz + \frac{1}{p}\int_\Omega(\xi(z)+\mu)|u|^pdz + \frac{1}{p}\int_{\partial\Omega}\beta(z)|u|^pd\sigma - \int_\Omega\hat{F}_\pm(z,u)dz\quad
	\mbox{for all}\ u\in W^{1,p}(\Omega).$$

Also, let $\varphi:W^{1,p}(\Omega)\rightarrow\RR$ be the energy (Euler) functional for problem (\ref{eq1}) defined by
$$\varphi(u)=\int_\Omega G(Du)dz + \frac{1}{p}\int_\Omega\xi(z)|u|^pdz + \frac{1}{p}\int_{\partial\Omega}\beta(z)|u|^pd\sigma - \int_\Omega F(z,u)dz\quad \mbox{for all}\ u\in W^{1,p}(\Omega).$$
This functional is also locally Lipschitz.
\begin{proposition}\label{prop7}
	If hypotheses $H(a)(i)(ii)(iii)$, $H(\xi)$, $H(\beta)$, $H_1$ hold, then the functionals $\varphi$ and $\hat{\varphi}_\pm$ are coercive.
\end{proposition}
\begin{proof}
	We  {present} the proof for the functional $\varphi$, the proofs for $\hat{\varphi}_\pm$ are similar.
	
	We proceed by contradiction. So, suppose that $\varphi$ is not coercive. Then we can find $\{u_n\}_{n\geqslant1}\subseteq W^{1,p}(\Omega)$ and $c_6>0$ such that
	\begin{equation}\label{eq10}
		\|u_n\|\rightarrow+\infty\ \mbox{and}\ \varphi(u_n)\leqslant c_6\ \mbox{for all}\ n\in\NN.
	\end{equation}
	
	We have
	\begin{equation}\label{eq11}
		\int_\Omega G(Du_n)dz + \frac{1}{p}\int_\Omega\xi(z)|u_n|^pdz + \frac{1}{p}\int_{\partial\Omega}\beta(z)|u_n|^pd\sigma - \int_\Omega F(z,u_n)dz\leqslant c_6\quad \mbox{for all}\ n\in\NN.
	\end{equation}
	
	We set $y_n=\frac{u_n}{\|u_n\|},\ n\in\NN$. Then $\|y_n\|=1$ for all $n\in\NN$ and so by passing to a suitable subsequence if necessary, we may assume that
	\begin{equation}\label{eq12}
		y_n\xrightarrow{w}y\ \mbox{in}\ W^{1,p}(\Omega)\ \mbox{and}\ y_n\rightarrow y\ \mbox{in}\ L^p(\Omega)\ \mbox{and}\ \mbox{in}\ L^p(\partial\Omega).
	\end{equation}
	
	From (\ref{eq11}) and Corollary 3, we have
	\begin{equation}\label{eq13}
		\frac{c_1}{p(p-1)}\|Dy_n\|^p_p + \frac{1}{p}\int_\Omega\xi(z)|y_n|^pdz + \frac{1}{p}\int_\Omega\beta(z)|y_n|^pd\sigma - \int_\Omega\frac{F(z,u_n)}{\|u_n\|^p}dz \leqslant \frac{c_6}{\|u_n\|^p}\quad
		 \mbox{for all}\ n\in\NN.\nonumber
	\end{equation}
	
	From hypothesis $H_1(i)$ and Lebourg's nonsmooth mean value theorem (see Clarke [14], p. 41), we obtain
	\begin{eqnarray*}
		|F(z,x)|\leqslant c_7(1+|x|^p)\ \mbox{for almost all}\ z\in\RR\ \mbox{and some}\ c_7>0,\\
		\Rightarrow\left\{\frac{F(\cdot,u_n(\cdot))}{\|u_n\|^p}\right\}_{n\geqslant1} \subseteq L^1(\Omega)\ \mbox{is uniformly integrable (see}\ (\ref{eq12})).
	\end{eqnarray*}
	
	Then the Dunford-Pettis theorem and hypothesis $H_1(ii)$ imply (at least for a subsequence) that
	\begin{equation}\label{eq14}
		\frac{F(\cdot,u_n(\cdot))}{\|u_n\|^p}\xrightarrow{w}\frac{c_1}{p(p-1)}\eta|y|^p\ \mbox{in}\ L^1(\Omega)
	\end{equation}
	\begin{equation*}
		\mbox{with}\ \eta\in L^\infty(\Omega), \eta(z) \leqslant\hat{\lambda}_1(p,\hat{\xi},\hat{\beta})\ \mbox{for almost all}\ z\in\Omega.
	\end{equation*}
	
	So, passing to the limit as $n\rightarrow\infty$ in (\ref{eq13}) and using (\ref{eq10}), (\ref{eq12}), (\ref{eq14}), we obtain
	\begin{eqnarray}
		& & \frac{c_1}{p(p-1)} \|Dy\|^p_p + \frac{1}{p}\int_\Omega\xi(z)|y|^pdz + \frac{1}{p}\int_{\partial\Omega}\beta(z)|y|^pd\sigma \leqslant \frac{c_1}{p(p-1)}\int_\Omega\eta|y|^pdz, \nonumber \\
		& \Rightarrow & \frac{c_1}{p(p-1)}\left[\|Dy\|^p_p + \int_\Omega\hat{\xi}(z)|y|^pdz + \int_{\partial\Omega}\hat{\beta}(z)|u|^pd\sigma\right] \leqslant \frac{c_1}{p(p-1)}\int_\Omega\eta|y|^pdz, \nonumber \\
		& \Rightarrow & \|Dy\|_p^p + \int_\Omega\hat{\xi}(z)|y|^pdz + \int_{\partial\Omega}\beta(z)|y|^pd\sigma \leqslant \int_\Omega\eta(z)|y|^p dz. \label{eq15}
	\end{eqnarray}
	
	First, we assume that $\eta\not\equiv\hat{\lambda}_1(p,\hat{\xi},\hat{\beta})$. Then from (\ref{eq15}) and Lemma 4.11 of Mugnai and Papageorgiou [30], we see that we can find $c_8>0$ such that
	 $c_8\|y\|^p\leqslant0$
	It follows that $y=0$.
	We deduce that
	\begin{eqnarray*}
		& & \|Dy_n\|_p\rightarrow0,\\
		& \Rightarrow & y_n\rightarrow0\ \mbox{in}\ W^{1,p}(\Omega)\ \mbox{(see (\ref{eq12}))},
	\end{eqnarray*}
	which contradicts the fact that $\|y_n\|=1$ for all $n\in\NN$.
	
	Now, we assume that $\eta(z)=\hat{\lambda}_1(p,\hat{\xi},\hat{\beta})$ for almost all $z\in\Omega$. Then from (\ref{eq15}) and (\ref{eq7}) we have
	\begin{eqnarray*}
		& & \|Dy\|^p_p + \int_\Omega\hat{\xi}(z)|y|^pdz + \int_{\partial\Omega}\hat{\beta}(z)|y|^pd\sigma = \hat{\lambda}_1(p,\hat{\xi},\hat{\beta})\|y\|_p^p,\\
		& \Rightarrow & y = \tau\hat{u}_1(p,\hat{\xi},\hat{\beta})\ \mbox{for some}\ \tau\in\RR.
	\end{eqnarray*}
	
	If $\tau=0$, then $y=0$ and as above, we have
	$y_n\rightarrow0\ \mbox{in}\ W^{1,p}(\Omega),$
	which contradicts the fact that $\|y_n\|=1$ for all $n\in\NN$.
	Then $\tau\neq0$ and in order to fix things, we assume that $\tau>0$ (the reasoning is similar if $\tau<0$). Since $\hat{u}_1(p,\hat{\xi},\hat{\beta})\in D_+$, we have
	$y(z)>0$ for all $z\in\Omega$.
	It follows that
	\begin{equation}\label{eq16}
		u_n(z)\rightarrow+\infty\ \mbox{for all}\ z\in\Omega.
	\end{equation}
	
	Hypothesis $H_1(iii)$ implies that given any $\hat{\mu}>0$, we can find $c_9=c_9(\mu)>0$ such that
	\begin{equation}\label{eq17}
		vx-pF(z,x)\geqslant\hat{\mu}\ \mbox{for almost all}\ z\in\Omega,\ \mbox{all}\ x\geqslant c_9,\ \mbox{all}\ v\in\partial F(z,x).
	\end{equation}
	
	For almost all $z\in\Omega$, the function $s\mapsto\frac{F(z,s)}{s^p}$ is locally Lipschitz on $[c_9,+\infty)$. So, using Proposition 2.3.14 of Clarke [14], p. 48, we have
	\begin{equation}\label{eq18}
		\partial\left(\frac{F(z,s)}{s^p}\right) \subseteq \frac{s^p\partial F(z,s)-ps^{p^{-1}}F(z,s)}{s^{2p}}.
	\end{equation}
	
	By Rademacher's theorem (see Gasinski and Papageorgiou [13], Theorem 3.120, p. 433) we know that for almost all $z\in\Omega$ the function $s\mapsto\frac{F(z,s)}{s^p}$ is differentiable for almost all $s\in[c_9,\infty)$ and at every such point $s\in[c_9,\infty)$ of differentiability, we have
	\begin{equation*}
		\frac{d}{ds}\left(\frac{F(z,s)}{s^p}\right)\in\partial\left(\frac{F(z,s)}{s^p}\right)
	\end{equation*}
	(see Clarke [14], Theorem 2.5.1, p. 63). So, from (\ref{eq18}) we see that we can find $v\in\partial F(z,s)$ such that
	\begin{eqnarray}
		&\displaystyle \frac{d}{ds}\left(\frac{F(z,s)}{s^p}\right) & = \frac{s^pv-ps^{p-1}F(z,s)}{s^{2p}}  = \frac{sv-pF(z,s)}{s^{p+1}} \nonumber \\
		& & \geqslant \frac{\hat{\mu}}{s^{p+1}}\ \mbox{for almost all}\ z\in\Omega,\ \mbox{all}\ s\in[c_9,\infty)\ (see (\ref{eq17})), \nonumber \\
		&\displaystyle \Rightarrow \frac{F(z,y)}{y^p} - \frac{F(z,x)}{x^p} & \geqslant - \frac{\hat{\mu}}{p}\left[\frac{1}{y^p}-\frac{1}{x^p}\right]\ \mbox{for almost all}\ z\in\Omega,\ all\ y\geqslant x\geqslant c_9. \label{eq19}
	\end{eqnarray}
	
	From hypothesis $H_1(ii)$ and using Lebourg's mean value theorem (see Clarke [14], Theorem 2.3.7, p. 41), we obtain
	\begin{equation}\label{eq20}
		\limsup_{x\rightarrow+\infty}\frac{F(z,x)}{x^p}\leqslant\frac{c_1}{p(p-1)}\hat{\lambda}_1(p,\hat{\xi},\hat{\beta})\ \mbox{uniformly for almost all}\ z\in\Omega.
	\end{equation}
	
	In (\ref{eq19}) we pass to the limit as $y\rightarrow+\infty$ and use (\ref{eq20}). Then
	\begin{eqnarray}
		& & \frac{c_1}{p(p-1)}\hat{\lambda}_1(p,\hat{\xi}, \hat{\beta}) - \frac{F(z,x)}{x^p} \geqslant \frac{\hat{\mu}}{p}\frac{1}{x^p}\ \mbox{for almost all}\ z\in\Omega,\ \mbox{all}\ x\geqslant c_9, \nonumber \\
		& \Rightarrow & \frac{c_1}{p-1}\hat{\lambda}_1(p,\hat{\xi}, \hat{\beta})x^p - pF(z,x) \geqslant \hat{\mu}\ \mbox{for almost all}\ z\in\Omega,\ all\ x\geqslant c_9. \label{eq21}
	\end{eqnarray}
	
	Recall that $\hat{\mu}>0$ is arbitrary. Then it follows from (\ref{eq21}) that
	\begin{equation}\label{eq22}
		\lim_{x\rightarrow+\infty}\left[\frac{c_1}{p-1}\hat{\lambda}_1(p,\hat{\xi}, \hat{\beta})x^p - pF(z,x)\right] = +\infty\ \mbox{uniformly for almost all}\ z\in\Omega.
	\end{equation}
	
	From (\ref{eq16}), (\ref{eq22}) and Fatou's lemma, we have
	\begin{equation}\label{eq23}
		\lim_{n\rightarrow\infty}\int_\Omega\left[\frac{c_1}{p-1}\hat{\lambda}_1(p,\hat{\xi}, \hat{\beta})|u_n(z)|^p-pF(z,u_n(z))\right]dz = +\infty .
	\end{equation}
	
	From (\ref{eq11}) and Corollary 3, we have
	\begin{eqnarray}
		& & \frac{c_1}{p(p-1)}\left[\|Du_n\|_p^p + \int_\Omega\hat{\xi}(z)|u_n|^pdz + \int_{\partial\Omega}\hat{\beta}(z)|u_n|^pd\sigma\right] - \nonumber \\
		& & \int_\Omega F(z,u_n)dz \leqslant c_6\ \mbox{for all}\ n\in\NN, \nonumber \\
		& \Rightarrow & \int_\Omega\left[\frac{c_1}{p-1}\hat{\lambda}_1(p,\hat{\xi}, \hat{\beta})|u_n|^p - pF(z,u_n)\right]dz\leqslant pc_6\ \mbox{for all}\ n\in\NN\ \mbox{(see (\ref{eq7}))}. \label{eq24}
	\end{eqnarray}
	
	Comparing (\ref{eq23}) and (\ref{eq24}), we reach a contradiction. This proves that $\varphi$ is coercive. Similarly, we show the coercivity of the functionals $\hat{\varphi}_\pm$.
\end{proof}

\begin{corollary}\label{cor8}
	If hypotheses $H(a)(i)(ii)(iii), H(\xi),H(\beta)$ and $H_1$ hold, then the functionals $\varphi$ and $\hat{\varphi}_\pm$ satisfy the nonsmooth $C$-condition.
\end{corollary}
Using Proposition 7 and global minimization of $\hat{\varphi}_\pm$ (the direct method of calculus of variations), we can produce two nontrivial smooth solutions of constant sign.
\begin{proposition}\label{prop9}
	If hypotheses $H(a), H(\xi),H(\beta)$ and $H_1$ hold, then problem (\ref{eq1}) was at least two nontrivial constant sign smooth solutions
	$$u_0\in D_+\ and\ w_0\in-D_+,$$
	which are local minimizers of $\varphi$.
\end{proposition}
\begin{proof}
	First we establish the existence of a positive solution.
	
	From Proposition \ref{prop7} we know that $\hat{\varphi}_+$ is coercive. Using the Sobolev embedding theorem and the compactness of the trace map, we see that $\hat{\varphi}_+$ is sequentially weakly lower semicontinuous. So, by the Weierstrass-Tonelli theorem, we know that we can find $u_0\in W^{1,p}(\Omega)$ such that
	\begin{equation}\label{eq25}
		\hat{\varphi}_+(u_0)=\inf[\hat{\varphi}_+(u):u\in W^{1,p}(\Omega)].
	\end{equation}
	
	Hypotheses $H(a)(iv)$ and $H_1(iv)$ imply that given $\epsilon>0$, we can find $\delta=\delta(\epsilon)\in\left(0,1\right]$ such that
	\begin{eqnarray}
		&&G(y)\leq\frac{1}{q}[\tilde{c}+\epsilon]\, |y|^q\ \mbox{for all}\ y\in\RR^N\ \mbox{with}\ |y|\leq\delta,\label{eq26}\\
		&&(\vartheta_0(z)-\epsilon)x^{q-1}\leq v\ \mbox{for almost all}\ z\in\Omega,\ \mbox{all}\ 0\leq x\leq\delta,\ \mbox{all}\ v\in\partial F(z,x).\label{eq27}
	\end{eqnarray}
	
	Recall that $\hat{u}_1(p,\tilde{\xi},\tilde{\beta})\in D_+$. Therefore we can find $t\in(0,1)$ small such that
	\begin{equation}\label{eq28}
		t\hat{u}_1(q,\tilde{\xi},\tilde{\beta})(z)\in[0,\delta]\ \mbox{and}\ t|D\hat{u}_1(q,\tilde{\xi}^+,\tilde{\beta})(z)|\leq\delta\ \mbox{for all}\ z\in\Omega.
	\end{equation}
	
	Then we have
	\begin{eqnarray}\label{eq29}
		\hat{\varphi}_+(t\hat{u}_1(q,\tilde{\xi},\tilde{\beta}))&\leq&\frac{t^q}{q}[\tilde{c}+\epsilon]||D\hat{u}_1(q,\tilde{\xi}^+,\tilde{\beta})||^q_q+\frac{1}{p}\int_{\Omega}\xi(z)|t\hat{u}_1(q,\tilde{\xi}^+,\tilde{\beta})|^pdz\nonumber\\
		&&+\frac{1}{p}\int_{\partial\Omega}\beta(z)|t\hat{u}_1(q,\tilde{\xi}^+,\tilde{\beta})|^pd\sigma-\frac{1}{q}\int_{\Omega}(\vartheta_0(z)-\epsilon)|t\hat{u}_1(q,\tilde{\xi}^+,\tilde{\beta})|^qdz\nonumber\\
		&&(\mbox{see (\ref{eq8}), (\ref{eq26}), (\ref{eq27}), (\ref{eq28})})\nonumber\\
		&\leq&\frac{t^q}{q}[\tilde{c}+\epsilon]||D\hat{u}_1(q,\tilde{\xi}^+,\tilde{\beta})||^q_q+\frac{t^q}{q}\int_{\Omega}\xi^+(z)\hat{u}_1(q,\tilde{\xi}^+,\tilde{\beta})^qdz\nonumber\\
		&&+\frac{t^q}{q}\int_{\partial\Omega}\beta(z)\hat{u}_1(q,\tilde{\xi}^+,\tilde{\beta})^qd\sigma-\frac{t^q}{q}\left[\int_{\Omega}\vartheta_0(z)\hat{u}_1(q,\tilde{\xi}^+,\tilde{\beta})^qdz-\epsilon\right]\nonumber\\
		&&(\mbox{recall that}\ 0<\delta\leq 1,\, q\leq p\ \mbox{and}\ ||\hat{u}_1(q,\tilde{\xi}^+,\tilde{\beta})||_q=1)\nonumber\\
		&\leq&\frac{t^q}{q}\tilde{c}\left[||D\hat{u}_1(q,\tilde{\xi}^+,\tilde{\beta})||^q_q+\int_{\Omega}\tilde{\xi}^+(z)\hat{u}_1(q,\tilde{\xi}^+,\tilde{\beta})^qdz\right.\nonumber\\
		&&\left.+\int_{\partial\Omega}\beta(z)\hat{u}_1(q,\tilde{\xi}^+,\tilde{\beta})^qd\sigma\right]\nonumber\\
		&&+\frac{t^q}{q}\epsilon[\hat{\lambda}_1(q,\tilde{\xi}^+,\tilde{\beta})+1]-\frac{t^q}{q}\int_{\Omega}\vartheta_0(z)\hat{u}_1(q,\tilde{\xi}^+,\tilde{\beta})dz\nonumber\\
		&=&\frac{t^q}{q}\left[\int_{\Omega}(\tilde{c}\hat{\lambda}_1
(q,\tilde{\xi}^+,\tilde{\beta})-\vartheta_0(z))\hat{u}_1(q,\tilde{\xi}^+,\tilde{\beta})dz+
\epsilon(\hat{\lambda}_1(q,\tilde{\xi}^+,\tilde{\beta})+1)\right].
	\end{eqnarray}
	
	By hypothesis $H_1(iv)$ and since $\hat{u}_1(q,\tilde{\xi}^+,\tilde{\beta})\in D_+$, we have
	$$\hat{\mu}^*=\int_{\Omega}(\vartheta_0(z)-\tilde{c}
\hat{\lambda}_1(q,\tilde{\xi}^+,\tilde{\beta}))\hat{u}_1(q,\tilde{\xi}^+,\tilde{\beta})dz>0.$$
	
	Then from (\ref{eq29}) we see that
	\begin{equation}\label{eq30}
		\hat{\varphi}_+(t\hat{u}_1(q,\tilde{\xi}^+,\tilde{\beta}))\leq
\frac{t^q}{q}[-\hat{\mu}^*+\epsilon(\hat{\lambda}_1(q,\tilde{\xi}^+,\tilde{\beta})+1)].
	\end{equation}
	
	Choosing $\epsilon\in(0,1)$ small,  we infer from (\ref{eq30}) that
	$\hat{\varphi}_+(t\hat{u}_1(q,\tilde{\xi}^+,\tilde{\beta}))<0$, hence
		$\hat{\varphi}_+(u_0)<0=\hat{\varphi}_+(0)$ (\mbox{see (\ref{eq25})}). We deduce that
		$u_0\neq 0$.
	
	By (\ref{eq25}) we have
	\begin{eqnarray}\label{eq31}
		&&0\in\partial\hat{\varphi}_+(u_0),\nonumber\\
		&\Rightarrow&\left\langle A(u_0),h\right\rangle+\int_{\Omega}(\xi(z)+\mu)|u_0|^{p-2}u_0hdz+\int_{\partial\Omega}\beta(z)|u_0|^{p-2}u_0hd\sigma=\int_{\Omega}\hat{v}hdz
	\end{eqnarray}
	for all $h\in W^{1,p}(\Omega)$, with $v(z)\in\partial\hat{F}_+(z,u_0(z))$ for almost all $z\in\Omega$ (see Clarke [14], Theorem 2.7.3, p. 80).
	
	In (\ref{eq31}) we choose $h=-u^-_0\in W^{1,p}(\Omega)$. We obtain
	\begin{eqnarray*}
		&&\frac{c_1}{p-1}||Du^-_0||^p_p+\int_{\Omega}(\xi(z)+\mu)(u^-_0)^pdz\leq 0\\
		&&(\mbox{see Lemma \ref{lem2}, hypothesis $H(\beta)$ and (\ref{eq9})}),\\
		&\Rightarrow&c_{10}||u^-_0||^p\leq 0\ \mbox{for some}\ c_{10}>0\ (\mbox{recall that}\ \mu>||\xi||_{\infty}),\\
		&\Rightarrow&u_0\geq 0,u_0\neq 0.
	\end{eqnarray*}
	
	From (\ref{eq31}), (\ref{eq9}), Stampacchia's theorem (see Gasinski and Papageorgiou [13], Remark 2.4.16, p. 195) and Papageorgiou and R\u adulescu [31], we have
	\begin{eqnarray}\label{eq32}
		&&\left\{\begin{array}{ll}
			-{\rm div}\, a(Du_0(z))+\xi(z)u_0(z)^{p-1}\in\partial F(z,u_0(z))&\mbox{for almost all}\ z\in\Omega,\\
		\displaystyle	\frac{\partial u}{\partial n_a}+\beta(z)u^{p-1}_0=0&\mbox{on}\ \partial\Omega,
		\end{array}\right\}\\
		&\Rightarrow&u_0\ \mbox{is a positive solution of (\ref{eq1})}.\nonumber
	\end{eqnarray}
	
	From (\ref{eq32}) and Papageorgiou and R\u adulescu [28], we have $u_0\in L^{\infty}(\Omega)$. So, by the nonlinear regularity theory of Lieberman [18] (p. 320), we have
	\begin{equation}\label{eq33}
		u_0\in C_+\backslash\{0\}.
	\end{equation}
	
	Hypotheses $H_1(i),(iv)$ imply that given $\rho>0$, we can find $\tilde{\mu}_{\rho}>0$ such that
	\begin{equation}\label{eq34}
		v+\tilde{\mu}_{\rho}x^{p-1}\geq 0\ \mbox{for almost all}\ z\in\Omega\ \mbox{and all}\ 0\leq x\leq \rho,\ \mbox{all}\ v\in\partial F(z,x).
	\end{equation}
	
	Now let $\rho=||u_0||_{\infty}$ and let $\tilde{\mu}_{\rho}>0$ as postulated by (\ref{eq34}). Then from (\ref{eq32}) and (\ref{eq34}) we have
	\begin{eqnarray}\label{eq35}
		&&-{\rm div}\, a(Du_0(z))+(\xi(z)+\tilde{\mu}_{\rho})u_0(z)^{p-1}\geq 0\ \mbox{for almost all}\ z\in\Omega,\nonumber\\
		&\Rightarrow&{\rm div}\, a(Du_0(z))\leq[||\xi||_{\infty}+\tilde{\mu}_{\rho}]u_0(z)^{p-1}\ \mbox{for almost all}\ z\in\Omega.
	\end{eqnarray}
	
	Let $\mu_0(t)=ta_0(t)$ for all $t>0$. We have
	\begin{eqnarray}\label{eq36}
		&&t\mu'_0(t)=t^2a'_0(t)+ta_0(t)\geq c_{11}t^{p-1}\ \mbox{for some}\ c_{11}>0\ \mbox{and all}\ t>0\\
		&&\mbox{(see hypotheses H(a)(i),(iii) and (\ref{eq4}))}.\nonumber
	\end{eqnarray}
	
	Integrating by parts, we obtain
	\begin{equation}\label{eq37}
		\int^t_0s\mu'_0(s)ds=t\mu_0(t)-\int^t_0\mu_0(s)ds
		=t^2a_0(t)-G_0(t)
		\geq\frac{c_{11}}{p}t^p\ (\mbox{see (\ref{eq36})}).
	\end{equation}
	
	Let $H(t)=t^2a_0(t)-G_0(t)$ and $H_0(t)=\frac{c_{11}}{p}t^p$ for all $t>0$. Pick $\delta\in(0,1)$ and $s>0$ and consider the sets
	$$C_1=\{t\in(0,1):H(t)\geq s\},\ C_2=\{t\in(0,1):H_0(t)\geq s\}.$$
	
	From (\ref{eq37}) we see that
	$C_2\subseteq C_1$, hence
		$\inf C_1\leq \inf C_2$.
	So, from Gasinski and Papageorgiou [29] (Proposition 1.55, p. 12),  we have
	$H^{-1}(s)\leq H^{-1}_0(s).$
	
	Then for $\mu^*_{\rho}=||\xi||_{\infty}+\tilde{\mu}_{\rho}$, we have
	\begin{equation}\label{eq38}
		\int^{\delta}_{0}\frac{1}{H^{-1}\left(\frac{\mu^*_{\rho}}{p}s^p\right)}ds
\geq\int^{\delta}_{0}\frac{1}{H^{-1}_0\left(\frac{\mu^*_{\rho}}{p}s^p\right)}ds=
\frac{\mu^*_{\rho}}{c_{11}}\int^{\delta}_0\frac{ds}{s}=+\infty.
	\end{equation}
	
	Inequalities (\ref{eq35}), (\ref{eq38}) permit the use of the nonlinear strong maximum principle of Pucci and Serrin [19], p. 111. We deduce that
	$u_0(z)>0\ \mbox{for all}\ z\in\Omega.$
	Invoking the boundary point of the theorem of Pucci and Serrin [19] (p. 120),  we obtain that
	$u_0\in D_+.$
	
	Note that $\hat{\varphi}_+|_{C_+}=\varphi|_{C_+}$ (see (\ref{eq8})). So, it follows that
	$u_0$ is a local $C^1(\overline{\Omega})$-minimizer of $\varphi$ (\mbox{see (\ref{eq25})}). By Proposition \ref{prop4}, we deduce that
		$u_0$ is a local $W^{1,p}(\Omega)$-minimizer of $\varphi$.
	
	Similarly, working with the functional $\hat{\varphi}_-$ we obtain a negative solution $w_0\in-D_+$ which is a local minimizer of $\varphi$.
\end{proof}

In fact, we can show that problem (\ref{eq1}) admits extremal constant sign solutions, that is, there are a smallest positive solution $\hat{u}_+\in D_+$ and a biggest negative solution $\hat{u}_-\in D_+$.

To this end, note that hypotheses $H_1(i),(iv)$ imply that given $\epsilon>0$ and $r\in(p,p^*)$ (recall $p^*=\left\{\begin{array}{ll}
	\frac{Np}{N-p}&\mbox{if}\ p<N\\
	+\infty&\mbox{if}\ N\leq p
\end{array}\right.$, the critical Sobolev exponent), we can find $c_{12}=c_{12}(\epsilon,r)>0$ such that
\begin{eqnarray}\label{eq39}
	&&vx\geq(\vartheta_0(z)-\epsilon)|x|^q-c_{12}|x|^r\\
	&&\mbox{for almost all}\ z\in\Omega\ \mbox{and all}\ x\in\RR,\ v\in\partial F(z,x).\nonumber
\end{eqnarray}

This unilateral growth estimate for $\partial F(z,\cdot)$ leads to the following auxiliary nonlinear Robin problem
\begin{eqnarray}\label{eq40}
	&&\left\{\begin{array}{l}
		-{\rm div}\, a(Du(z))+\xi^+(z)|u(z)|^{p-1}u(z)=(\vartheta_0(z)-\epsilon)|u(z)|^{q-2}u(z)-c_{12}|u(z)|^{r-2}u(z)\
		\mbox{in}\ \Omega,\\
	\displaystyle	\frac{\partial u}{\partial n_a}+\beta(z)|u|^{p-2}u=0\ \mbox{on}\ \partial\Omega.
	\end{array}\right\}
\end{eqnarray}
\begin{proposition}\label{prop10}
	If hypotheses $H(a),H(\xi),H(\beta)$ hold, then for $\epsilon>0$ small problem (\ref{eq40}) has a unique positive solution
	$$\tilde{u}^*\in D_+$$
	and since (\ref{eq40}) is odd, $\tilde{v}^*=-\tilde{v}^*\in-D_+$ is the unique negative solution of (\ref{eq40}).
\end{proposition}
\begin{proof}
	First, we prove the existence of a positive solution for problem (\ref{eq40}). For this purpose, we introduce the $C^1$-functional $\tilde{\psi}_+:W^{1,p}(\Omega)\rightarrow\RR$ defined by
	\begin{eqnarray*}
		&&\tilde{\psi}_+(u)=\int_{\Omega}G(Du)dz+\frac{1}{p}\int_{\Omega}\xi^+(z)|u|^pdz+\frac{1}{p}||u^-||^p_p+\frac{1}{p}\int_{\partial\Omega}\beta(z)|u|^pd\sigma\\
		&&-\frac{1}{q}\int_{\Omega}(\vartheta_0(z)-\epsilon)(u^+)^qdz+\frac{c_{12}}{r}||u^+||^r_r\ \mbox{for all}\ u\in W^{1,p}(\Omega).
	\end{eqnarray*}
	
	Using Corollary \ref{cor3}, we have
	\begin{eqnarray}\label{eq41}
		\tilde{\psi}_+(u)&\geq&\frac{c_1}{p(p-1)}||Du^-||^p_p+\frac{1}{p}\int_{\Omega}\xi^+(z)(u^-)^pdz+\frac{1}{p}||u^-||^p_p\nonumber\\
		&&+\frac{c_1}{p(p-1)}||Du^+||^p_p+\frac{1}{p}\int_{\Omega}\xi^+(z)(u^+)^pdz+\frac{c_{12}}{r}||u^+||^r_r-c_{13}||u^+||^q\\
		&&\mbox{for some}\ c_{13}>0\ (\mbox{since}\ \vartheta_0\in L^{\infty}(\Omega)).\nonumber
	\end{eqnarray}
	
	Recall that $q\leq p<r$. So, from (\ref{eq41}) it follows that $\tilde{\psi}_+$ is coercive. Also, it is sequentially weakly lower semicontinuous. So, by the Weierstrass-Tonelli theorem we can find $\tilde{u}^*\in W^{1,p}(\Omega)$ such that
	\begin{equation}\label{eq42}
		\tilde{\psi}_+(\tilde{u}^*)=\inf[\tilde{\psi}_+(u):u\in W^{1,p}(\Omega)].
	\end{equation}
	
	As in the proof of Proposition \ref{prop9}, and since $q\leq p<r$, we show that
	$\tilde{\psi}_+(\tilde{u}^*)<0=\tilde{\psi}_+(0)$, hence
		$\tilde{u}^*\neq 0$.
	
	Also, from (\ref{eq42}) we have
	\begin{eqnarray}\label{eq43}
		&&\tilde{\psi}'_+(\tilde{u}^*)=0,\nonumber\\
		&\Rightarrow&\left\langle A(\tilde{u}^*),h\right\rangle+\int_{\Omega}\xi^+(z)|\tilde{u}^*|^{p-2}\tilde{u}^*hdz-\int_{\Omega}((\tilde{u}^*)^-)^{p-1}hdz+\int_{\partial\Omega}\beta(z)((\tilde{u}^*)^+)^{p-1}hd\sigma\nonumber\\
		&&-\int_{\Omega}(\vartheta_0(z)-\epsilon)((\tilde{u}^*)^+)^{q-1}hdz+c_{12}\int_{\Omega}((\tilde{u}^*)^+)^{r-1}hdz=0\ \mbox{for all}\ h\in W^{1,p}(\Omega).
	\end{eqnarray}
	
	In (\ref{eq43}) we choose $h=-(\tilde{u}^*)^-\in W^{1,p}(\Omega)$. Then
	\begin{eqnarray*}
		&&\frac{c_1}{p-1}||D(\tilde{u}^*)^-||^p_p+||(\tilde{u}^*)^-||^p_p\leq 0\ (\mbox{see Lemma \ref{lem2}}),\\
		&\Rightarrow&\tilde{u}^*\geq 0,\ \tilde{u}^*\neq 0.
	\end{eqnarray*}
	
	So, equation (\ref{eq43}) becomes
	\begin{eqnarray*}
		&&\left\langle A(\tilde{u}^*),h\right\rangle+\int_{\Omega}\xi^+(z)(\tilde{u}^*)^{p-1}hdz+\int_{\partial\Omega}\beta(z)((\tilde{u}^*)^+)^{p-1}hd\sigma\\
		&&-\int_{\Omega}(\vartheta_0(z)-\epsilon)(\tilde{u}^*)^{q-1}hdz+c_{12}\int_{\Omega}(\tilde{u}^*)^{r-1}hdz=0\ \mbox{for all}\ h\in W^{1,p}(\Omega),\\
		&\Rightarrow&\tilde{u}^*\ \mbox{is a positive solution of problem (\ref{eq40})}.
	\end{eqnarray*}
	
	As before (see the proof of Proposition \ref{prop9}), using the nonlinear regularity theory and the nonlinear maximum principle, we have
	$\tilde{u}^*\in D_+.$
	
	Next, we show the uniqueness of this positive solution. To this end we introduce the integral functional $j:L^q(\Omega)\rightarrow\bar{\RR}=\RR\cup\{+\infty\}$ defined by
	$$j(u)=\left\{\begin{array}{ll}		\displaystyle\int_{\Omega}G(Du^{1/q})dz+\frac{1}{p}\int_{\Omega}\xi^+(z)u^{p/q}dz+\frac{1}{p}\int_{\partial\Omega}\beta(z)u^{p/q}d\sigma\ \mbox{if}\ u\geq 0,\ u^{1/q}\in W^{1,p}(\Omega)&\\
		+\infty\ \ \mbox{otherwise}.&
	\end{array}\right.$$
	
	Suppose that $u_1,u_2\in\{u\in L^q(\Omega):j(u)<\infty\}$ (the effective domain of $j(\cdot)$). Let $y_1=u_1^{1/q}$ and $y_2=u_2^{1/q}$. By definition $y_1,y_2\in W^{1,p}(\Omega)$. We set
	$y=[tu_1+(1-t)u_2]^{1/q}\ \mbox{with}\ t\in[0,1].$
	Then $y\in W^{1,p}(\Omega)$ and using Lemma 1 of Diaz and Saa [32], we have
	\begin{eqnarray*}
		&&|Dy(z)|\leq[t|Dy_1(z)|^q+(1-t)|Dy_2(z)|^q]^{1/q},\\
		&\Rightarrow&G_0(|Dy(z)|)\leq G_0([t|Dy_1(z)|^q+(1-t)|Dy_2(z)|^q]^{1/q})\\
		&&(\mbox{since}\ G_0(\cdot)\ \mbox{is increasing})\\
		&&\leq tG_0(|Dy_1(z)|)+(1-t)G_0(|Dy_2(z)|)\ \mbox{for almost all}\ z\in\Omega\\
		&&(\mbox{see hypothesis}\ H(a)(iv)),\\
		&\Rightarrow&G(Dy(z))\leq tG(Du_1(z)^{1/q})+(1-t)G(Du_2(z)^{1/q})\ \mbox{for almost all}\ z\in\Omega,\\
		&\Rightarrow&u\mapsto\int_{\Omega}G(Du^{1/q})dz\ \mbox{is convex}.
	\end{eqnarray*}
	
	Since $\xi^+\geq 0,\ \beta\geq 0$ (see hypothesis $H(\beta)$) and $q\leq p$, we see that
	$${\rm dom}\, j\ni u\mapsto\int_{\Omega}\xi^+(z)u^{p/q}dz+\int_{\partial\Omega}\beta(z)u^{p/q}d\sigma\ \mbox{is convex}.$$
	
	Therefore the integral functional $j(\cdot)$ is convex. By Fatou's lemma, it is also lower semicontinuous.
	
	Suppose that $u_1,u_2$ are two positive solutions of (\ref{eq40}). From the first part of the proof, we have
	$u_1,u_2\in D_+.$
	Then for $h\in C^1(\overline{\Omega})$ and  small $|t|>0$ we have
	$u_1^q+th,\ u_2^q+th\in {\rm dom}\, j.$
	
	Since $j(\cdot)$ is convex, we can easily check that $j(\cdot)$ is G\^ateaux differentiable at $u_1^q$ and at $u_2^q$ in the direction $h$. Moreover, using the nonlinear Green's identity (see Gasinski and Papageorgiou [13], p. 210), we have
	\begin{eqnarray*}
		&&j'(u_1^q)(h)=\frac{1}{q}\int_{\Omega}\frac{-{\rm div}\,a (Du_1)+\xi^+(z)u_1^{p-1}}{u_1^{q-1}}hdz\\
		&&j'(u_2^q)(h)=\frac{1}{q}\int_{\Omega}\frac{-{\rm div}\,a (Du_2)+\xi^+(z)u_2^{p-1}}{u_2^{q-1}}hdz\ \mbox{for all}\ h\in C^1(\overline{\Omega}).
	\end{eqnarray*}
	
	Since $j(\cdot)$ is convex, $j'(\cdot)$ is monotone. Therefore
	\begin{eqnarray*}
		0&\leq&\int_{\Omega}\left(\frac{-{\rm div}\,a(Du_1)+\xi^+(z)u_1^{q-1}}{u_1^{q-1}}-\frac{-{\rm div}\, a(Du_2)+\xi^+(z)u_2^{q-1}}{u_2^{q-1}}\right)(u_1^q-u_2^q)dz\\
		&=&c_{12}\int_{\Omega}(u_2^{r-q}-u_1^{r-q})(u_1^q-u_2^q)dz\leq 0\ (\mbox{recall}\ r>q),\\
		\Rightarrow&&u_1=u_2.
	\end{eqnarray*}
	
	This proves the uniqueness of the positive solution $\tilde{u}^*\in D_+$ of (\ref{eq40}).
	
	Problem (\ref{eq40}) is odd. Therefore $\tilde{v}^*=-\tilde{u}^*\in -D_+$ is the unique negative solution of (\ref{eq40}).
\end{proof}

Let $S_+$ be the set of positive solutions of (\ref{eq1}) and $S_-$ the set of negative solutions of (\ref{eq1}). From Proposition \ref{prop9} and its proof, we know that
$S_+\neq\emptyset$, $S_+\subseteq D_+$ and $S_-\neq\emptyset,\ S_-\subseteq-D_+.$
Moreover, as in Filippakis and Papageorgiou [33] (see also Papageorgiou, R\u adulescu and Repov\v s [34]), we have that
\begin{center}
$S_+$ is downward directed (that is, if $u_1,u_2\in S_+$, then we can find $u\in S_+$ such that $u\leq u_1,\ u\leq u_2$);\\
$S_-$ is upward directed (that is, if $v_1,v_2\in S_-$, then we can find $v\in S_-$ such that $v_1\leq v,\ v_2\leq v$).
\end{center}
\begin{proposition}\label{prop11}
	If hypotheses $H(a),H(\xi), H(\beta),H_1$ hold, then $\tilde{u}^*\leq u$ for all $u\in S_+$ and $v\leq\tilde{v}^*$ for all $v\in S_-$.
\end{proposition}
\begin{proof}
	Let $u\in S_+$ and consider the following Carath\'eodory function
	\begin{eqnarray}\label{eq44}
		\tau_+(z,x)=\left\{\begin{array}{ll}
			0&\mbox{if}\ x<0\\
			(\vartheta_0(z)-\epsilon)x^{q-1}-c_{12}x^{r-1}+x^{p-1}&\mbox{if}\ 0\leq x\leq u(z)\\
			(\vartheta_0(z)-\epsilon)u(z)^{q-1}-c_{12}u(z)^{q-1}+u(z)^{p-1}&\mbox{if}\ u(z)<x.
		\end{array}\right.
	\end{eqnarray}
	
	Let $T_+(z,x)=\int^x_0\tau_+(z,s)ds$ and consider the $C^1$-functional $\chi_+:W^{1,p}(\Omega)\rightarrow\RR$ defined by
	$$\chi_+(u)=\int_{\Omega}G(Du)dz+\frac{1}{p}\int_{\Omega}(\xi^+(z)+1)|u|^{p}dz+\frac{1}{p}
\int_{\partial\Omega}\beta(z)|u|^pd\sigma-\int_{\Omega}T_+(z,u)dz\
		\mbox{for all}\ u\in W^{1,p}(\Omega).$$
	
	Corollary \ref{cor3} and (\ref{eq44}) imply that $\chi_+$ is coercive. Also, it is sequentially weakly lower semicontinuous. So, by the Weierstrass-Tonelli theorem we can find $\hat{u}^*\in W^{1,p}(\Omega)$ such that
	\begin{equation}\label{eq45}
		\chi_+(\hat{u}^*)=\inf[\chi_+(u):u\in W^{1,p}(\Omega)].
	\end{equation}
	
	As in the proof of Proposition \ref{prop9}, using hypothesis $H_1(iv)$, we see that
	$\chi_+(\hat{u}^*)<0=\chi_+(0)$, hence
		$\hat{u}^*\neq 0$.
	
	From (\ref{eq45}) we have
	\begin{eqnarray}\label{eq46}
		&&\chi'_+(\hat{u}^*)=0,\nonumber\\
		&\Rightarrow&\left\langle A(\hat{u}^*),h\right\rangle+\int_{\Omega}(\xi^+(z)+1)|\hat{u}^*|^{p-2}\hat{u}^*hdz+\int_{\partial\Omega}\beta(z)|\hat{u}^*|^{p-2}\hat{u}^*hd\sigma\nonumber\\
		&&=\int_{\Omega}\tau_+(z,\hat{u}^*)hdz\ \mbox{for all}\ h\in W^{1,p}(\Omega).
	\end{eqnarray}
	
	In (\ref{eq46}) we first choose $h=-(\hat{u}^*)^-\in W^{1,p}(\Omega)$. Using Lemma \ref{lem2} and (\ref{eq44}), we obtain
	\begin{eqnarray*}
		&&\frac{c_1}{p-1}||D(\hat{u}^*)^-||^p_p+||(\hat{u}^*)^-||^p_p\leq 0,\\
		&\Rightarrow&\hat{u}^*\geq 0,\ \hat{u}^*\neq 0.
	\end{eqnarray*}
	
	Next, in (\ref{eq46}) we choose $h=(\hat{u}^*-u)^+\in W^{1,p}(\Omega)$. Using (\ref{eq44}), we obtain
	\begin{eqnarray}\label{eq47}
		&&\left\langle A(\hat{u}^*),(\hat{u}^*-u)^+\right\rangle+\int_{\Omega}(\xi^+(z)+1)(\hat{u}^*)^{p-1}(\hat{u}^*-u)^+dz\nonumber\\
		&&+\int_{\partial\Omega}\beta(z)(\hat{u}^*)^{p-1}(\hat{u}^*-u)^+hd\sigma\nonumber\\
		&&=\int_{\Omega}(\vartheta_0(z)-\epsilon)u^{q-1}(\hat{u}^*-u)^+dz-c_{12}\int_{\Omega}u^{r-1}(\hat{u}^*-u)^+dz+\int_{\Omega}u^{p-1}(\hat{u}^*-u)^+dz.
	\end{eqnarray}
	
	Since $u\in S_+$, we can find $v\in L^{p'}(\Omega)\ \left(\frac{1}{p}+\frac{1}{p'}=1\right)$ such that
	\begin{eqnarray}\label{eq48}
		\left\{\begin{array}{l}
			v(z)\in\partial F(z,u(z))\ \mbox{for almost all}\ z\in\Omega,\\
			-{\rm div}\,a(Du(z))+\xi(z)u(z)^{p-1}=v(z)\ \mbox{for almost all}\ z\in\Omega,\\
	\displaystyle		\frac{\partial u}{\partial n_a}+\beta(z)u^{p-1}=0\ \mbox{on}\ \partial\Omega.
		\end{array}\right\}
	\end{eqnarray}
	
	Using (\ref{eq39}), (\ref{eq47}), (\ref{eq48}), we have
	\begin{eqnarray*}
		&&\left\langle A(\hat{u}^*),(\hat{u}^*-u)^+\right\rangle+\int_{\Omega}(\xi^+(z)+1)(\hat{u}^*)^{p-1}(\hat{u}^*-u)^+dz+\\
		&&\int_{\partial\Omega}\beta(z)(\hat{u}^*)^{p-1}(\hat{u}^*-u)^+d\sigma\leq\int_{\Omega}v(\hat{u}^*-u)^+dz+\int_{\Omega}u^{p-1}(\hat{u}^*-u)^+dz\\
		&&=\left\langle A(u),(\hat{u}^*-u)^+\right\rangle+\int_{\Omega}(\xi(z)+1)u^{p-1}(\hat{u}^*-u)^+dz+\int_{\partial\Omega}\beta(z)u^{p-1}(\hat{u}^*-u)^+d\sigma\\
		&&(\mbox{since}\ u\in S_+)\\
		&&\leq\left\langle A(u),(\hat{u}^*-u)^+\right\rangle+\int_{\Omega}(\xi^+(z)+1)u^{p-1}(\hat{u}^*-u)^+dz+\int_{\partial\Omega}\beta(z)u^{p-1}(\hat{u}^*-u)^+d\sigma,\\
		&\Rightarrow&\left\langle A(\hat{u}^*)-A(u),(\hat{u}^*-u)^+\right\rangle+\int_{\Omega}(\xi^+(z)+1)((\hat{u}^*)^{p-1}-u^{p-1})(\hat{u}^*-u)^+dz\\
		&&+\int_{\partial\Omega}\beta(z)((\hat{u}^*)^{p-1}-u^{p-1})(\hat{u}^*-u)^+d\sigma\leq 0,\\
		&\Rightarrow&\hat{u}^*\leq u\ (\mbox{see Lemma \ref{lem2} and hypothesis}\ H(\beta)).
	\end{eqnarray*}
	
	Therefore we have proved that
	\begin{equation}\label{eq49}
		\hat{u}^*\in[0,u],\ \hat{u}^*\neq 0.
	\end{equation}
	
	It follows from (\ref{eq44}) and (\ref{eq49}) that
	$\hat{u}^*\ \mbox{is a positive solution of (\ref{eq40})}$, hence
		$\hat{u}^*=\tilde{u}^*\in D_+$ (\mbox{see Proposition \ref{prop10}}). We conclude that
		$\tilde{u}^*\leq u$ for all $u\in S_+$.
	A similar argument shows that
	$v\leq\tilde{v}^*$ for all $v\in S_-.$
\end{proof}

Now we can produce extremal constant sign solutions for problem (\ref{eq1}). These solutions play an important role in the argument of Section 4, where we establish the existence of nodal solutions.
\begin{proposition}\label{prop12}
	If hypotheses $H(a),H(\xi),H(\beta),H_1$ hold, then problem (\ref{eq1}) admits a smallest positive solution $\hat{u}_+\in D_+$ and a biggest negative solution $\hat{u}_-\in D_+$.
\end{proposition}
\begin{proof}
	Recalling that $S_+$ is downward directed and using Lemma 3.10 of Hu and Papageorgiou [35] (p. 178), we can find a decreasing sequence $\{u_n\}_{n\geq 1}\subseteq S_+$ such that
	$\inf\limits_{n\geq 1}u_n=\inf S_+.$
	Evidently, $\{u_n\}_{n\geq 1}\subseteq W^{1,p}(\Omega)$ is bounded. So, we may assume that
	\begin{equation}\label{eq50}
		u_n\stackrel{w}{\rightarrow}\hat{u}_+\ \mbox{in}\ W^{1,p}(\Omega)\ \mbox{and}\ u_n\rightarrow\hat{u}_+\ \mbox{in}\ L^p(\Omega)\ \mbox{and in}\ L^p(\partial\Omega).
	\end{equation}
	
	For every $n\in\NN$, we can find $v_n\in L^{p'}(\Omega)$ such that
	\begin{eqnarray}
		&&v_n(z)\in\partial F(z,u_n(z))\ \mbox{for almost all}\ z\in\Omega\label{eq51}\\
		&&\left\langle A(u_n),h\right\rangle+\int_{\Omega}\xi(z)u_n^{p-1}hdz+\int_{\partial\Omega}\beta(z)u_n^{p-1}hd\sigma=\int_{\Omega}v_nhdz\ \mbox{for all}\ h\in W^{1,p}(\Omega).\label{eq52}
	\end{eqnarray}
	
	From (\ref{eq51}), (\ref{eq50}) and hypothesis $H_1(i)$, we see that $\{v_n\}_{n\geq 1}\subseteq L^{p'}(\Omega)$ is bounded. So, we may also assume that
	\begin{equation}\label{eq53}
		v_n\stackrel{w}{\rightarrow}\hat{v}\ \mbox{in}\ L^{p'}(\Omega)\ \mbox{as}\ n\rightarrow\infty.
	\end{equation}
	
	From (\ref{eq50}), (\ref{eq53}) and Proposition 3.9 of Hu and Papageorgiou [35] (p. 694), we have
	\begin{equation}\label{eq54}
		\hat{v}(z)\in\partial F(z,\hat{u}_+(z))\ \mbox{for almost all}\ z\in\Omega.
	\end{equation}
	
	In (\ref{eq52}) we choose $h=u_n-\hat{u}_+\in W^{1,p}(\Omega)$, pass to the limit as $n\rightarrow\infty$ and use (\ref{eq50}), (\ref{eq53}). Then we obtain
	\begin{eqnarray}\label{eq55}
		&&\lim\limits_{n\rightarrow\infty}\left\langle A(u_n),u_n-\hat{u}_+\right\rangle=0,\nonumber\\
		&\Rightarrow&u_n\rightarrow\hat{u}_+\ \mbox{in}\ W^{1,p}(\Omega)\ (\mbox{see Proposition \ref{prop5}}).
	\end{eqnarray}
	
	From Proposition \ref{prop11} we have
	\begin{eqnarray}\label{eq56}
		&&\tilde{u}^*\leq u_n\ \mbox{for all}\ n\in\NN\nonumber\\
		&\Rightarrow&\tilde{u}^*\leq\hat{u}_+.
	\end{eqnarray}
	
	In (\ref{eq52}) we pass to the limit as $n\rightarrow\infty$ and use (\ref{eq53}) and (\ref{eq55}). Then
	\begin{eqnarray*}
		&&\left\langle A(\hat{u}_+),h\right\rangle+\int_{\Omega}\xi(z)\hat{u}_+^{p-1}hdz+\int_{\partial\Omega}\beta(z)\hat{u}_+^{p-1}hd\sigma=\int_{\Omega}\hat{v}hdz\ \mbox{for all}\ h\in W^{1,p}(\Omega),\\
		&\Rightarrow&\hat{u}_+\in S_+\ \mbox{(see (\ref{eq54}), (\ref{eq56})) and}\ \hat{u}_+=\inf S_+.
	\end{eqnarray*}
	
	Similarly, we  {can} produce  $\hat{u}_-\in S_-$ such that $\hat{u}_-=\sup S_+.$
\end{proof}

\section{Nodal Solutions}

In this section we produce a nodal (sign changing) solution for problem (\ref{eq1}). The idea is simple. Having the extremal constant sign solutions $\hat{u}_{\pm}$ (see Proposition \ref{prop12}), by suitable truncation-perturbation techniques we restrict ourselves to the order interval $[\hat{u}_-,\hat{u}_+]$. Using variational arguments (Theorem \ref{th1}) and Proposition \ref{prop6}, we show that problem (\ref{eq1}) has a solution $y_0\in[\hat{u}_-,\hat{u}_+]\backslash\{0,\hat{u}_{\pm}\}$. The extremality of $\hat{u}_{\pm}$ implies that $y_0$ is nodal. For this strategy to work, we need to strengthen the hypotheses on the primitive $F(z,x)$. So,  {we now} assume that
$F(z,x)=\int^x_0f(z,s)ds$,
with the integrand $f(z,s)$ being measurable with possible jump discontinuities in $s\in\RR$. We set
$$f_l(z,x)=\liminf\limits_{x'\rightarrow x}f(z,x')\ \mbox{and}\ f_u(z,x)=\limsup\limits_{x'\rightarrow x} f(z,x').$$

The precise hypotheses on the integrand $f(z,x)$ are the following;

\smallskip
$H_2:$ $f:\Omega\times\RR\rightarrow\RR$ is a measurable function, for almost all $z\in\Omega$, $f(z,\cdot)$ is continuous at $x=0$, it has only jump discontinuities and
\begin{itemize}
	\item[(i)] $|f(z,x)|\leq a(z)(1+|x|^{p-1})$ for almost all $z\in\Omega$\ \mbox{and all}\ $x\in\RR$, with $a\in L^{\infty}(\Omega)$;
	\item[(ii)] $\limsup\limits_{x\rightarrow\pm\infty}\frac{f_u(z,x)}{|x|^{p-2}x}\leq\frac{c_1}{p-1}\hat{\lambda}_1(p,\hat{\xi},\hat{\beta})$ uniformly for almost all $z\in\Omega$;
	\item[(iii)] $\lim\limits_{x\rightarrow+\infty}[f_l(z,x)x-pF(z,x)]=\lim\limits_{x\rightarrow-\infty}[f_u(z,x)x-pF(z,x)]=+\infty$ uniformly for almost all $z\in\Omega$;
	\item[(iv)] there exists $c^*>\tilde{c}\hat{\lambda}_2(q,\tilde{\xi}^+,\tilde{\beta})$ such that
	$$c^*\leq\liminf\limits_{x\rightarrow 0^+}\frac{f_l(z,x)}{x^{p-1}},\ \liminf\limits_{x\rightarrow 0^-}\frac{f_u(z,x)}{|x|^{p-2}x}\ \mbox{uniformly for almost all}\ z\in\Omega.$$
\end{itemize}
\begin{remark}
	The above hypotheses imply that the primitive
	$F(z,x)=\int^x_0f(z,s)ds$
	is a locally Lipschitz integrand and
	$\partial F(z,x)=[f_l(z,x),f_u(z,x)]$
	(see Clarke [14], p. 34  and Chang [15]). Therefore hypotheses $H_2$ are a more restrictive version of hypotheses $H_1$ used in Section 3.
\end{remark}
	
	Using these new stronger conditions on the reaction term, we now prove the existence of nodal (that is, sign changing) solutions.
\begin{proposition}\label{prop13}
	If hypotheses $H(a),H(\xi),H(\beta),H_2$ hold, then problem (\ref{eq1}) admits a nodal solution
	$$y_0\in[\hat{u}_-,\hat{u}_+]\cap C^1(\overline{\Omega}).$$
\end{proposition}	
\begin{proof}
	Let $\hat{u}_+\in D_+$ and $\hat{u}_-\in-D_+$ be the two extremal constant sign solutions of problem (\ref{eq1}) produced in Proposition \ref{prop12}. Then we can find $v_+,v_-\subseteq L^{p'}(\Omega)\ \left(\frac{1}{p}+\frac{1}{p'}=1\right)$ such that
	\begin{eqnarray}
		&&v_+(z)\in\partial F(z,\hat{u}_+(z)),\ v_-(z)\in\partial F(z,\hat{u}_-(z))\ \mbox{for almost all}\ z\in\Omega,\label{eq57}\\
		&&\left\langle A(\hat{u}_+,h)\right\rangle+\int_{\Omega}\xi(z)\hat{u}_+^{p-1}hdz+\int_{\partial\Omega}\beta(z)\hat{u}_+^{p-1}hd\sigma=\int_{\Omega}v_+hdz,\label{eq58}\\
		&&\left\langle A(\hat{u}_-,h)\right\rangle+\int_{\Omega}\xi(z)|\hat{u}_-|^{p-2}\hat{u}_-hdz+\int_{\partial\Omega}\beta(z)|\hat{u}_-|^{p-2}\hat{u}_-hd\sigma=\int_{\Omega}v_-hdz\label{eq59}\\
		&&\mbox{for all}\ h\in W^{1,p}(\Omega).\nonumber
	\end{eqnarray}
	
	As before, let $\mu>||\xi||_{\infty}$ (see hypothesis $H(\xi)$) and consider the following measurable functions
	\begin{eqnarray}
		&&\hat{f}_+(z,x)=\left\{\begin{array}{ll}
			0&\mbox{if}\ x<0\\
			f(z,x)+\mu x^{p-1}&\mbox{if}\ 0\leq x\leq\hat{u}_+(z)\\
			v_+(z)+\mu\hat{u}_+(z)^{p-1}&\mbox{if}\ \hat{u}_+(z)<x,
		\end{array}\right.\label{eq60}\\
		&&\hat{f}_-(z,x)=\left\{\begin{array}{ll}
			v_-(z)+\mu|\hat{u}_-(z)|^{p-2}\hat{u}_-(z)&\mbox{if}\ x<\hat{u}_-(z)\\
			f(z,x)+\mu|x|^{p-2}x&\mbox{if}\ \hat{u}_-(z)\leq x\leq 0\\
			0&\mbox{if}\ 0<x,
		\end{array}\right.\label{eq61}\\
		&&\hat{f}(z,x)=\left\{\begin{array}{ll}
			v_-(z)+\mu|\hat{u}_-(z)|^{p-2}\hat{u}_-(z)&\mbox{if}\ x<\hat{u}_-(z)\\
			f(z,x)+\mu|x|^{p-2}x&\mbox{if}\ \hat{u}_-(z)\leq x\leq\hat{u}_+(z)\\
			v_+(z)+\mu\hat{u}_+(z)^{p-1}&\mbox{if}\ \hat{u}_+(z)<x.
		\end{array}\right.\label{eq62}
	\end{eqnarray}
	
	Let $\hat{F}_{\pm}(z,x)=\int^x_0\hat{f}_{\pm}(z,s)ds$ and $\hat{F}(z,x)=\int^x_0\hat{f}(z,s)ds$.
	
	We also consider  the corresponding truncation of the boundary term. So, we consider the Carath\'eodory function
	\begin{equation}\label{eq63}
		\hat{\beta}(z,x)=\left\{\begin{array}{ll}
			\beta(z)|\hat{u}_-(z)|^{p-2}\hat{u}_-(z)&\mbox{if}\ x<\hat{u}_-(z)\\
			\beta(z)|x|^{p-2}x&\mbox{if}\ \hat{u}_-(z)\leq x\leq\hat{u}_+(z)\\
			\beta(z)\hat{u}_+(z)^{p-1}&\mbox{if}\ \hat{u}_+(z)<x
		\end{array}\right.\ \mbox{for all}\ (z,x)\in\partial\Omega\times\RR\,.
	\end{equation}
	
	We set $\hat{B}(z,x)=\int^x_0\hat{\beta}(z,s)ds$ for all $(z,x)\in\partial\Omega\times\RR$.
	
	We introduce the locally Lipschitz functionals $\hat{\psi}_{\pm},\hat{\psi}:W^{1,p}(\Omega)\rightarrow\RR$ defined by
	\begin{eqnarray*}
		&&\hat{\psi}_{\pm}(u)=\int_{\Omega}G(Du)dz+\frac{1}{p}\int_{\Omega}(\xi(z)+\mu)|u|^pdz+\int_{\partial\Omega}\hat{B}(z,\pm u^{\pm})d\sigma-\int_{\Omega}\hat{F}_{\pm}(z,u)dz\\
		&&\hat{\psi}(u)=\int_{\Omega}G(Du)dz+\frac{1}{p}\int_{\Omega}(\xi(z)+\mu)|u|^pdz+\int_{\partial\Omega}\hat{B}(z,u)d\sigma-\int_{\Omega}\hat{F}(z,u)dz\\
		&&\mbox{for all}\ u\in W^{1,p}(\Omega).
	\end{eqnarray*}
	
	Also, we consider the following order intervals in $W^{1,p}(\Omega)$
	$T_+=[0,\hat{u}_+]$, $T_-=[\hat{u}_-,0]$, $T=[\hat{u}_-,\hat{u}_+].$
	\begin{claim}\label{claim1}
		$K_{\hat{\psi}}\subseteq T,\ K_{\hat{\psi}_+}=\{0,\hat{u}_+\},\ K_{\hat{\psi}_-}=\{0,\hat{u}_-\}$.
	\end{claim}
	
	Let $u\in K_{\hat{\psi}}$. We have
	\begin{eqnarray}\label{eq64}
		&&0\in\partial\hat{\psi}(u),\nonumber\\
		&\Rightarrow&\left\langle A(u),h\right\rangle+\int_{\Omega}(\xi(z)+\mu)|u|^{p-2}uhdz+\int_{\partial\Omega}\hat{\beta}(z,u)hd\sigma=\int_{\Omega}\hat{v}hdz
	\end{eqnarray}
	for all $h\in W^{1,p}(\Omega)$, with $\hat{v}\in L^{p'}(\Omega),\hat{v}(z)\in\partial\hat{F}(z,u(z))$ for almost all $z\in\Omega$.
	
	In (\ref{eq64}) we choose $h=(u-\hat{u}_+)^+\in W^{1,p}(\Omega)$. Using (\ref{eq62}) and (\ref{eq63}), we have
	\begin{eqnarray*}
		&&\left\langle A(u),(u-\hat{u}_+)^+\right\rangle+\int_{\Omega}(\xi(z)+\mu)u^{p-1}(u-\hat{u}_+)^+dz+\int_{\partial\Omega}\beta(z)\hat{u}_+^{p-1}(u-\hat{u}_+)^+d\sigma\\
		&&=\int_{\Omega}(v_+(z)+\mu\hat{u}_+^{p-1})(u-\hat{u}_+)^+dz\\
		&&=\left\langle A(\hat{u}_+),(u-\hat{u}_+)^+\right\rangle+\int_{\Omega}(\xi(z)+\mu)\hat{u}^{p-1}_+(u-\hat{u}_+)^+dz+\int_{\partial\Omega}\beta(z)\hat{u}^{p-1}_+(u-\hat{u}_+)^+d\sigma\\
		&&(\mbox{see (\ref{eq57}), (\ref{eq58})}),\\
		&\Rightarrow&\left\langle A(u)-A(\hat{u}_+),(u-\hat{u}_+)^+\right\rangle+\int_{\Omega}(\xi(z)+\mu)(u^{p-1}-\hat{u}_+^{p-1})(u-\hat{u}_+)^+dz=0,\\
		&\Rightarrow&u\leq\hat{u}\ (\mbox{recall that}\ \mu>||\xi||_{\infty}).
	\end{eqnarray*}
	
	In a similar fashion, choosing $h=(\hat{u}_--u)^+\in W^{1,p}(\Omega)$ in (\ref{eq64}) and using (\ref{eq59}) and (\ref{eq61}), we obtain
	$\hat{u}_-\leq u$, hence
		$u\in T$. We deduce that
		$K_{\hat{\psi}}\subseteq T$.
	Arguing similarly, we show that
	$K_{\hat{\psi}_+}\subseteq T_+$ and $K_{\hat{\psi}_-}\subseteq T_-.$
	
	From (\ref{eq57}), (\ref{eq60}), (\ref{eq61}) and the extremality of the solutions $\hat{u}_+\in D_+$ and $\hat{u}_-\in-D_+$, we conclude that
	$K_{\hat{\psi}_+}=\{0,\hat{u}_+\}$ and $K_{\hat{\psi}_-}=\{0,\hat{u}_-\}.$
	This proves Claim \ref{claim1}.
	\begin{claim}\label{claim2}
		$\hat{u}_+$ and $\hat{u}_-$ are local minimizers of $\hat{\psi}$.
	\end{claim}
	
	It is clear from Corollary \ref{cor3}, (\ref{eq60}) and the fact that $\mu>||\xi||_{\infty}$, that $\hat{\psi}_+$ is coercive. Moreover, using the Sobolev embedding theorem and the compactness of the trace map, we see that $\hat{\psi}_+$ is sequentially weakly lower semicontinuous. So, by the Weierstrass-Tonelli theorem, we can find $\tilde{u}_+\in W^{1,p}(\Omega)$ such that
	\begin{equation}\label{eq65}
		\hat{\psi}_+(\tilde{u}_+)=\inf[\hat{\psi}_+(u):u\in W^{1,p}(\Omega)]=\hat{m}_+
	\end{equation}
	
	Hypothesis $H_2(iv)$, as in the proof of Proposition \ref{prop9}, implies that
	$\hat{m}_+=\hat{\psi}_+(\tilde{u}_+)<0=\hat{\psi}_+(0)$, hence
		$\tilde{u}_+\not=0$.
		We deduce that $\tilde{u}_+=\hat{u}_+$ (see Claim \ref{claim1} and recall that $\tilde{u}_+\in K_{\hat{\psi}_+}$, see (\ref{eq65})).

	Note that $\hat{\psi}_+|_{C_+}=\hat{\psi}|_{C_+}$ (see (\ref{eq60}), (\ref{eq62})) and recall that $\hat{u}_+\in D_+$. So, it follows that
	$\hat{u}_+$ is a local $C^1(\overline{\Omega})$-minimizer of $\hat{\psi}$. By Proposition \ref{prop4} we deduce that
		$\hat{u}_+$ is a local $W^{1,p}(\Omega)$-minimizer of $\hat{\psi}$.

	Similarly for $\hat{u}_-\in-D_+$, using this time the functional $\hat{\psi}_-$ and (\ref{eq61}).
	
	This proves Claim \ref{claim2}.
	
	Without any loss of generality, we may assume that
	$\hat{\psi}(\hat{u}_-)\leq\hat{\psi}(\hat{u}_+).$
	The reasoning is similar if the opposite inequality holds.
	
	We assume that $K_{\hat{\psi}}$ is finite. Otherwise on account of (\ref{eq62}), Claim \ref{claim1} and the extremality of $\hat{u}_{\pm}$, we already have an infinity of nodal solutions. Claim \ref{claim2} implies that we can find  small $\rho\in(0,1)$ such that
	\begin{equation}\label{eq66}
\hat{\psi}(\hat{u}_-)\leq\hat{\psi}(\hat{u}_+)<\inf[\hat{\psi}(u):||u-\hat{u}_+||=\rho]=\hat{m}_{\rho},\ \rho<||\hat{u}_--\hat{u}_+||
	\end{equation}
	(see Aizicovici, Papageorgiou and Staicu [36], proof of Proposition 29). From (\ref{eq62}) and Corollary \ref{cor3} it is clear that
	\begin{equation}\label{eq67}
 {		\hat{\psi}\ \mbox{is coercive}
		\Rightarrow\hat{\psi}\ \mbox{satisfies the nonsmooth C-condition}.   }
	\end{equation}
	
	Then (\ref{eq66}) and (\ref{eq67}) permit the use of Theorem \ref{th1}. Therefore we can find $y_0\in W^{1,p}(\Omega)$ such that
	\begin{equation}\label{eq68}
		y_0\in K_{\hat{\psi}}\subseteq T\ \mbox{(see Claim \ref{claim1}) and}\ \hat{m}_{\rho}\leq\hat{\psi}(y_0).
	\end{equation}
	
	It follows from (\ref{eq63}), (\ref{eq64}) and (\ref{eq62}) that
	$y_0\ \mbox{is a solution of (\ref{eq1}) and}\ y_0\notin\{\hat{u}_+,\hat{u}_-\}.$
	
	Evidently, if we show that $y_0\neq 0$, then $y_0$ is a nodal solution of (\ref{eq1}) (see (\ref{eq68})). From Theorem \ref{th1}, we have
	\begin{eqnarray}\label{eq69}
	    &&\hat{\psi}(y_0)=\inf\limits_{\gamma\in\Gamma}\max\limits_{0\leq t\leq 1}\hat{\psi}(\gamma(t))
	\end{eqnarray}
with $\Gamma=\{\gamma\in C([0,1],W^{1,p}(\Omega)):\gamma(0)=\hat{u}_-,\gamma(1)=\hat{u}_+\}$.

According to (\ref{eq69}), to show the nontriviality of $y_0$, it suffices to produce a path $\gamma_*\in\Gamma$ such that $\hat{\psi}|_{\gamma_*}<0$. In what follows, we construct such a path in $\Gamma$.

Recall (see Section 2) that
$\partial B^{L^q}_1=\{u\in L^p(\Omega):||u||_q=1\}$ and $M=W^{1,p}(\Omega)\cap \partial B^{L^q}_1.$
Also, we define
$M_c=M\cap C^1(\overline{\Omega}).$

We consider the following two sets of paths:
\begin{eqnarray*}
&&\hat{\Gamma}=\{\hat{\gamma}\in C([-1,1],M):\hat{\gamma}(-1)=-\hat{u}_1(q,\tilde{\xi}^+,\tilde{\beta}),\hat{\gamma}(1)=\hat{u}_1(q,\tilde{\xi}^+,\tilde{\beta})\},\\
&&\hat{\Gamma}_c=\{\hat{\gamma}\in C([-1,1],M_c):\hat{\gamma}(-1)=-\hat{u}_1(q,\tilde{\xi}^+,\tilde{\beta}),\hat{\gamma}(1)=
\hat{u}_1(q,\tilde{\xi}^+,\tilde{\beta})\}.
\end{eqnarray*}
    From Papageorgiou and R\u adulescu [37], we know that
    $\hat{\Gamma}_c\ \mbox{is dense in}\ \hat{\Gamma}.$
    Then Proposition \ref{eq6} implies that given $\hat{\delta}>0$, we can find $\hat{\gamma}_0\in\hat{\Gamma}_c$ such that
    \begin{equation}\label{eq70}
    \max\limits_{0\leq t\leq 1}\tilde{c}\tilde{k}^+_q(\hat{\gamma}_0(t))\leq\tilde{c}\hat{\lambda}_2(q,\tilde{\xi}^+,\tilde{\beta})+\hat{\delta}
    \end{equation}
with $\tilde{k}^+_q(u)=||Du||^q_q+\int_{\Omega}\tilde{\xi}^+(z)|u|^pd\sigma+\int_{\partial\Omega}\tilde{\beta}(z)|u|^pd\sigma$ for all $u\in W^{1,p}(\Omega)$.

Hypothesis $H(a)(iv)$ implies that given $\epsilon>0$, we can find $\delta_1=\delta_1(\epsilon)\in(0,1)$ such that
\begin{equation}\label{eq71}
G(y)\leq\frac{\tilde{c}+\epsilon}{q}|y|^q\ \mbox{for all}\ y\in\RR^N\ \mbox{with}\ |y|\leq\delta_1.
\end{equation}

Hypothesis $H_2(iv)$ say that we can find $c^*_1\in(\tilde{c}\hat{\lambda}_2(q,\tilde{\xi}^+,\tilde{\beta}),c^*)$ and $\delta_2=\delta_2(\epsilon)\in(0,1)$ such that
\begin{equation}\label{eq72}
\frac{1}{q}c^*_1|x|^q\leq F(z,x)\ \mbox{for almost all}\ z\in\Omega,\ \mbox{all}\ |x|\leq\delta_2.
\end{equation}

Let $0<\delta\leq \min\{\delta_1,\delta_2,\min\limits_{\overline{\Omega}}\hat{u}_+,\min\limits_{\overline{\Omega}}(-\hat{u}_-)\}$ (recall that $\hat{u}_+\in D_+,\hat{u}_-\in-D_+$). Since $\hat{\gamma}_0\in\hat{\Gamma}_c,\hat{u}_+\in D_+$ and $\hat{u}_-\in-D_+$, we can find $\vartheta\in(0,1)$ small such that
\begin{eqnarray}\label{eq73}
    &&\vartheta\hat{\gamma}_0(t)\in T=[\hat{u}_-,\hat{u}_+],\ \vartheta|\hat{\gamma}_0(t)(z)|\leq\delta,\ \vartheta|D\hat{\gamma}_0(t)(z)|\leq\delta\\
    &&\mbox{for all}\ t\in[-1,1],\ \mbox{all}\ z\in\overline{\Omega}.\nonumber
\end{eqnarray}

Then for $t\in[-1,1]$ we have
\begin{eqnarray}\label{eq74}
    \hat{\psi}(\vartheta\hat{\gamma}_0(t))&=&\int_{\Omega}G(\vartheta D\hat{\gamma}_0(t))dz+\frac{\vartheta^p}{p}\int_{\Omega}\xi(z)|\hat{\gamma}_0(t)|^pdz+\frac{\vartheta^p}{p}\int_{\partial\Omega}\beta(z)|\hat{\gamma}_0(t)|^pd\sigma-\nonumber\\
		&&\int_{\Omega}F(z,\vartheta\hat{\gamma}_0(t))dz\ (\mbox{see (\ref{eq62}), (\ref{eq63}), (\ref{eq68}) and recall the choice of}\ \delta>0)\nonumber\\
    &\leq&\frac{\tilde{c}+\epsilon}{q}\vartheta^q||D\hat{\gamma}_0(t)||^q_q+\frac{\tilde{c}\vartheta^q}{q}\int_{\Omega}\tilde{\xi}^+(z)|\hat{\gamma}_0(t)|^qdz+\nonumber\\
		&&\frac{\tilde{c}\vartheta^q}{q}\int_{\partial\Omega}\tilde{\beta}(z)|\hat{\gamma}_0(t)|^qd\sigma-\frac{c^*_1\vartheta^q}{q}||\hat{\gamma}_0(t)||^q_q\nonumber\\
    &&(\mbox{see (\ref{eq71}), (\ref{eq72}), (\ref{eq73}) and recall}\ \delta,\vartheta\in(0,1),q\leq p)\nonumber\\
    &\leq&\frac{\vartheta^q}{q}[(\tilde{c}\hat{\lambda}_2(q,\tilde{\xi}^+,\tilde{\beta})+\hat{\delta})+\epsilon c_{13}-c^*_1]\ \mbox{for some}\ c_{13}>0\\
    &&(\mbox{see (\ref{eq70}), (\ref{eq73}) and recall that}\ ||\hat{\gamma}_0(t)||_q=1\ \mbox{for some}\ c_{13}>0).\nonumber
\end{eqnarray}

We choose  small $\epsilon,\,\hat{\delta}>0$ so that
\begin{equation}\label{eq75}
    \hat{\psi}(\vartheta\hat{\gamma}_0(t))<0\ \mbox{for all}\ t\in[-1,1]\ (\mbox{see (\ref{eq74})}).
\end{equation}

We set $\gamma_0=\vartheta\hat{\gamma}_0$. This is a continuous path in $W^{1,p}(\Omega)$ (in fact in $C^1(\overline{\Omega})$), which connects $-\vartheta\hat{u}_1(q,\tilde{\xi}^+,\tilde{\beta})$ and $\vartheta \hat{u}_1(q,\tilde{\xi}^+,\tilde{\beta})$. Along this path we have
\begin{equation}\label{eq76}
\hat{\psi}|_{\gamma_0}<0.
\end{equation}

Recall that
$\hat{\psi}_+(\hat{u}_+)=\hat{m}_+<0=\hat{\psi}_+(0)$ (\mbox{see (\ref{eq65})}).
The functional $\hat{\psi}_+$ is coercive (see Corollary \ref{cor3} and (\ref{eq60})). So, it satisfies the nonsmooth C-condition. Let $\dot{\hat{\psi}}^{\eta}_+=\{u\in W^{1,p}(\Omega):\hat{\psi}_+(u)<\eta\}$ for any $\eta\in\RR$. Invoking the nonsmooth second deformation theorem of Corvellec [38], we can find a deformation $h:[0,1]\times\dot{\hat{\psi}}^0_+\rightarrow\dot{\hat{\psi}}^0_+$ such that
\begin{eqnarray}
    \bullet&& h(t,\cdot)|_{K^{\hat{m}_+}_{\hat{\psi}_+}}={\rm id}|_{K^{\hat{m}_+}_{\hat{\psi}_+}}\ (K^{\hat{m}_+}_{\hat{\psi}_+}=\{u\in K_{\hat{\psi}_+}:\hat{\psi}_+(u)=\hat{m}_+\});\label{eq77}\\
    \bullet&& h(1,\dot{\hat{\psi}}^0_+)\subseteq\dot{\hat{\psi}}^{\hat{m}_+}_+\cup K^{\hat{m}_+}_{\hat{\psi}_+}=\{\hat{u}_+\}\ (\mbox{see Claim \ref{claim1} and (\ref{eq65})});\label{eq78}\\
    \bullet&&\hat{\psi}_+(h(t,u))\leq\hat{\psi}_+(u)\ \mbox{for all}\ t\in[0,1],\ \mbox{all}\ u\in\dot{\hat{\psi}}^0_+.\label{eq79}
\end{eqnarray}

Let $\gamma_+(t)=h(t,\vartheta\hat{u}_1(q,\tilde{\xi}^+,\tilde{\beta}))^+$ for all $t\in[0,1]$. This is a well-defined path (see (\ref{eq76})) and of course it is continuous. Also, we have
\begin{itemize}
    \item $\gamma_+(0)=h(0,\vartheta\hat{u}_1(q,\tilde{\xi}^+,\tilde{\beta}))^+=\vartheta\hat{u}_1(q,
        \tilde{\xi}^+,\tilde{\beta})$
    (since $h$ is a deformation, see [13], Definition 4.11.2, p. 645);
    \item $\gamma_+(1)=h(1,\vartheta\hat{u}_1(q,\tilde{\xi}^+,\tilde{\beta}))^+=\hat{u}_+$ (see (\ref{eq78}) and recall $\hat{u}_+\in D_+$);
    \item $\hat{\psi}(\gamma_+(t))=\hat{\psi}_+(\gamma_+(t))\leq\hat{\psi}_+(\vartheta\hat{u}_1(q,\tilde{\xi}^+,\tilde{\beta}))=\hat{\psi}(\vartheta\hat{u}_1(q,\tilde{\xi}^+,\tilde{\beta})$ (see (\ref{eq60}), (\ref{eq62}), (\ref{eq79})).
\end{itemize}

So, $\gamma_+$ is a continuous path in $W^{1,p}(\Omega)$ connecting $\vartheta\hat{u}_1(q,\tilde{\xi}^+,\tilde{\beta})$ and $\hat{u}_+$ (see (\ref{eq77})), and along this path we have
\begin{equation}\label{eq80}
    \psi|_{\gamma_+}<0.
\end{equation}

Similarly, we produce another continuous path $\gamma_-$ in $W^{1,p}(\Omega)$ which connects $-\vartheta\hat{u}_1(q,\tilde{\xi}^+,\tilde{\beta})$ and $\hat{u}_-$ and along which we have
\begin{equation}\label{eq81}
    \psi|_{\gamma_-}<0.
\end{equation}

We concatenate $\gamma_-,\gamma_0,\gamma_+$ and generate $\gamma_*\in\Gamma$ such that
$\psi|_{\gamma_*}<0$ (\mbox{see (\ref{eq76}), (\ref{eq80}), (\ref{eq81})}). We conclude that
    $y_0$ is nodal.
As before (see the proof of Proposition \ref{prop9}), using the nonlinear regularity theory, we have
$y_0\in[\hat{u}_-,\hat{u}_+]\cap C^1(\overline{\Omega}).$
\end{proof}

Summarizing, we  have established the following multiplicity theorem for problem (\ref{eq1}).
\begin{theorem}\label{th14}
    If hypotheses $H(a),H(\xi),H(\beta),H_2$ hold, then problem (\ref{eq1}) has at least three nontrivial smooth solutions
    $$u_0\in D_+,\ v_0\in-D_+\ \mbox{and}\ y_0\in[v_0,u_0]\cap C^1(\overline{\Omega})\ \mbox{nodal}.$$
\end{theorem}

\begin{acknowledgements} This research was supported by the Slovenian Research Agency grants P1-0292, J1-8131, and J1-7025. V.D. R\u adulescu acknowledges the support through a grant of the Ministry of Research and Innovation, CNCS--UEFISCDI, project number PN-III-P4-ID-PCE-2016-0130, within PNCDI III.
\end{acknowledgements}

\end{document}